\documentclass[a4paper,USenglish,cleveref, autoref, thm-restate]{lipics-v2021}

\hideLIPIcs  %

\usepackage{tikz}
\usetikzlibrary{math,calc,arrows.meta,intersections,arrows,positioning,shapes,decorations.markings,decorations.pathreplacing,matrix,patterns}
\tikzstyle{vertex}=[circle,draw=black,fill=black,inner sep=0,minimum size=3pt,text=white,font=\footnotesize]
\usepackage{todonotes}
\usepackage{pifont}%

\newcommand{\eps}{\varepsilon}
\newcommand{\R}{{\mathbb R}}

\newcommand{\cG}{{\mathcal G}}
\newcommand{\cA}{{\mathcal A}}
\newcommand{\cL}{{\mathcal L}}
\newcommand{\norm}[1]{\left\|#1 \right\|}
\newcommand{\sqn}[1]{\norm{#1}^2}

\newcommand{\floor}[1]{\left\lfloor{#1}\right\rfloor}
\newcommand{\cbr}[1]{\left(#1\right)}
\newcommand{\br}[1]{\left\{#1\right\}}

\newcommand{\lmax}{\lambda_{\max}}

\newcommand{\lminp}{\lambda_{\min^+}}

\newcommand{\sumjn}{\sum_{j=1}^{n}}

\newcommand{\diag}[1]{\operatorname{diag}\cbr{#1}}

\makeatletter
\newcommand{\sref}[2]{\stackrel{\srefmultieq{#1}}{#2}}
\newcommand{\srefmultieq}[1]{%
  \begingroup
  \def\srefsep{}%
  \@for\srefitem:=#1\do{%
    \srefsep\eqref{\srefitem}%
    \def\srefsep{,}%
  }%
  \endgroup
}
\makeatother

\title{A Communication Complexity Lower Bound for Nonuniformly Convex Consensus Optimization}

\author{Demyan Yarmoshik}{MIRAI, Russia}{yarmoshik.d@miriai.org}{https://orcid.org/0000-0003-1912-1040}{}%

\author{Maxim Klimenko}{MIPT, Russia}{klimkomx@gmail.com}{}{}

\authorrunning{D. Yarmoshik and M. Klimenko} %

\Copyright{Demyan Yarmoshik and Maxim Klimenko} %



\category{} %

\relatedversion{} %

\nolinenumbers %

\EventEditors{John Q. Open and Joan R. Access}
\EventNoEds{2}
\EventLongTitle{42nd Conference on Very Important Topics (CVIT 2016)}
\EventShortTitle{CVIT 2016}
\EventAcronym{CVIT}
\EventYear{2016}
\EventDate{December 24--27, 2016}
\EventLocation{Little Whinging, United Kingdom}
\EventLogo{}
\SeriesVolume{42}
\ArticleNo{23}
\begin{document}

\maketitle

\begin{abstract}
We study the communication complexity of convex decentralized optimization over time-varying networks, where $n$ nodes hold private functions and must agree on the global minimizer using only synchronous exchanges with neighbors.
The cost is the number of communication rounds to reach accuracy $\varepsilon$ -- a measure akin to round complexity in the LOCAL model, but constrained by nodes sharing only oracle responses.
We prove a new lower bound of $\Omega\!\left(\chi_{\mathcal G} \sqrt{\kappa_g}\,\log\frac{n}{\chi_{\mathcal G}}\log\frac1\varepsilon\right)$ communication rounds, where $\chi_{\mathcal G}$ is the condition number of the network Laplacians and $\kappa_g$ that of the global objective, showing the round complexity attainable under uniform regularity cannot be matched in the nonuniform regime.
The construction rests on spectral graph theory: we embed time-rotating star gadgets into the edges of an expander and patch them to preserve spectral connectivity.
\end{abstract}

\section{Introduction}
Decentralized optimization is an active research direction in machine learning and distributed computing.
The area's core problem is to minimize the sum of locally stored functions $f_i$ 
\begin{equation}\label{eq:prob}
    \min_{x\in \R^d} \frac1n \sum_{i=1}^n f_i(x),
\end{equation}
by means of peer-to-peer communications.
These functions are stored by the nodes of the communication network, whose edges can change over time.
Problem \eqref{eq:prob} is usually called the consensus optimization problem, since all nodes must find the same global minimizer $x^*$ up to desired accuracy. 
The problem can, for example, represent federated learning tasks, in which multiple independent parties aim to train a common model's parameters $x$ on the combined set of data from all parties without actually sharing the data. 
In that case $f_i(x)$ is the empirical loss over the training samples in the $i$-th party's private dataset.%

Theoretical analysis of \textit{distributed algorithms} usually focuses on \textit{communication complexity}, which is the number of synchronous communication rounds (e.g. gradient transfers) between all adjacent computation nodes,
and \textit{computation complexity}, which is abstracted by the number of oracle calls (e.g. computations of gradients $\nabla f_i (x_i)$ at each node).

While the notion of communication complexity reveals similarities with the LOCAL model, in decentralized optimization nodes are not allowed to share all information about their functions $f_i$.
They can only exchange the oracle's responses to their queries and results of arithmetic manipulations of these responses.
We formally define the computation model in \Cref{sec:comp-model}.

So-called centralized distributed algorithms, which utilize a central parameter-server (star communication topology) or all-reduce communication schemes, converge at a rate depending on the global condition number $\kappa_g = L_g /\mu_g$ of the objective.
This parameter quantifies the ratio between the rates at which
the global objective $\frac1n \sum_{i=1}^n f_i(x)$ changes in different directions, using global smoothness $L_g$ and strong convexity $\mu_g$ parameters, see  definitions in \Cref{sec:defs}.
When a problem is parametrized by the global parameters $\kappa_g, L_g, \mu_g$, we call this the \textit{nonuniform} setup, since it does not impose uniform restrictions on properties of functions $f_i$.

In contrast, decentralized algorithms usually depend on the local condition number $\kappa_l = L_l /\mu_l$ which is the ratio of the worst (largest) smoothness parameter $L_l$ and the worst (smallest) strong convexity parameter $\mu_l$ over all local objective functions $f_i$. 
In general, a globally $\mu_g$-strongly convex objective can have non-convex local terms $f_i$, thus $\mu_l$ can be infinitely worse than $\mu_g$. 
The same holds for the smoothness parameters $L_g, L_l$, therefore the local condition number $\kappa_l$ can be arbitrarily larger than the global condition number $\kappa_g$, or can even be undefined.
This motivates the search for decentralized algorithms whose complexity is parametrized by $\kappa_g$ instead of $\kappa_l$.

\subsection{Contribution}
In this work we establish a new lower bound on the communication complexity of distributed algorithms over time-varying communication networks with smooth, strongly convex deterministic objectives.

\begin{theorem}\label{thm:main}
For any network condition number $\chi_\cG \geq 886$\footnote{In the complementary regime $\chi_\cG \in (1, 886)$ the factor $\chi_\cG$ in the bound is bounded by an absolute constant; this range can be covered by a separate construction tailored to this case, which we omit since it does not affect the asymptotic lower bound.}, global objective condition number $\kappa_g > 1$, and accuracy $\eps > 0$, there is a decentralized optimization problem on $n \geq 4\chi_\cG$ nodes such that the communication complexity of any distributed algorithm is lower bounded as
\begin{equation}\label{eq:lb_contrib}
   N_\cG = \Omega \left( \chi_\cG \sqrt{\kappa_g} \log \frac{n}{\chi_\cG} \log \frac{1}{\eps}\right).
 \end{equation} 
\end{theorem}

While there are decentralized algorithms that partially replace the dependence on $\kappa_l$ with $\kappa_g$, until recently there were no tighter communication complexity lower bounds for the nonuniform setup. 
A brief note \cite{KovKupRog26} announced an improved lower bound 
$N_G = \Omega \left( \sqrt{\chi_G \kappa_g} \log \frac{L_l}{L_g} \log \frac{1}{\eps}\right)$
for nonuniform decentralized optimization in static communication networks, showing that optimal complexity in uniform setup
$N_G = O \left( \sqrt{\chi_G \kappa_g} \log \frac{1}{\eps}\right)$ \cite{scaman2017optimal}
is not achievable in the nonuniform model.

We provide the same impossibility result for time-varying networks. 
Lower bound \eqref{eq:lb_contrib} can also be restated %
as $N_\cG = \Omega \left( \chi_\cG \sqrt {\kappa_g} \log \frac{L_l}{L_g} \log \frac{1}{\eps}\right)$,
matching the result of \cite{KovKupRog26} up to a square root of $\chi_G$, which corresponds to the optimal complexity
$N_\cG = O \left( \chi_\cG \sqrt {\kappa_g} \log \frac{1}{\eps}\right)$ of uniform decentralized optimization on time-varying graphs \cite{kovalev2021lower}.

The construction of our lower bound is based on expander graphs and spectral graph theory, which differs from classical lower bounds for the uniform setup that rely on trivial path and star graph constructions \cite{scaman2017optimal,kovalev2021lower,scaman2019optimal}.
The idea of using expander graphs to introduce a $\log n$ dependence on the number of nodes into communication lower bounds was also mentioned in a blogpost of S.~Bubeck\footnote{\url{https://blogs.princeton.edu/imabandit/2017/07/05/smooth-distributed-convex-optimization/}.} accompanying \cite{scaman2017optimal}.
That argument, however, only yields a lower bound of the form $\Omega(\max\{\sqrt{\chi_G}, \log n\})$, whereas both \cite{KovKupRog26} and the present work obtain the stronger multiplicative bounds $\Omega(\sqrt{\chi_G}\log n)$ and $\Omega(\chi_\cG \log n)$, respectively, which require more involved graph-theoretical analysis. %

\section{Related Work}
Minimizing communication complexity is in high demand among practitioners training large models.
The performance of this essentially distributed process is often limited by network capacity~\cite{lian2017can}, motivating the development of communication compression techniques~\cite{seide20141,richtarik2021ef21,beznosikov2023biased} and optimization of communication schemes \cite{boyd2004fastest,lakhotia2022polarfly,liao2025ub,ying2021exponential}.

Decentralized optimization is an approach to this problem which is rooted in gossip consensus algorithms~\cite{nedic2009distributed,nedic2010constrained,gorbunov2022recent}. 
The framework of {decentralized algorithms} utilizes \textit{communication matrices} (or \textit{gossip matrices}) $W$  
for the atomic communication operation of taking a weighted sum of vectors from immediate neighbors: $v_i = \sum_{j: (i, j) \in E_G} W_{ij} v_j$. 
This approach is favorable due to its simplicity of implementation and robustness, which allow it to be generalized to the time-varying setup, where edges in the communication graph can change between communication rounds~\cite{kovalev2021lower,metelev2023consensus,ryabinin2021moshpit}.

The standard model of communication complexity in decentralized optimization simply counts the number of communication rounds required to solve a problem for a given accuracy \cite{gorbunov2022recent,ye2023multi,scaman2017optimal}.
Up to privacy constraints, this is similar to the LOCAL model \cite{linial1992locality,rozhon2024local}.
Our model captures only a structural form of privacy: a node never shares its function $f_i$ directly, but exposes oracle responses. A stronger, quantitative notion is differential privacy \cite{dwork2006calibrating,dwork2014algorithmic}, which additionally bounds how much the shared responses can reveal about any individual data point and has been studied for decentralized learning \cite{allouah2024privacy}.
While no major works on decentralized optimization seem to mention the connection, there are explicit comparisons of the LOCAL and ``gossip'' models in the literature on distributed bandits and reinforcement learning \cite{dubey2020kernel,shihab2026locality,lazarsfeld2023simple}.
There are also more practical network models where communication time depends on heterogeneous link delays and/or bandwidths \cite{marfoq2020throughput,tyurin2026optimality}.

In the standard communication-round model there is a well-established complexity theory for both static and time-varying setups with matching upper and lower bounds in many setups \cite{scaman2017optimal,kovalev2021lower,scaman2019optimal,yarmoshik2025decentralized,koloskova2024optimization,tyurin2024optimal,li2024accelerated}.
Yet there is a gap in communication complexity in the nonuniform smooth and strongly convex setup. 
A prominent result here is the Mudag algorithm~\cite{ye2023multi}, which achieves $N_G = O\left(\sqrt{\chi_G \kappa_g} \left(\log{\frac{L_l}{L_g}} + \log{\frac{L_g}{\mu_g}}\right)\log{\frac{1}{\eps}}\right)$ communication complexity. %
Comparing with the optimal complexity of first-order decentralized algorithms in the uniform setup 
$O\left(\sqrt{\chi_G \kappa_l} \log{\frac{1}{\eps}}\right)$ \cite{scaman2017optimal,kovalev2020optimal}, we see that Mudag is optimal up to the $\log{\frac{L_l}{L_g}} + \log{\frac{L_g}{\mu_g}}$ term.
It was recently shown \cite{KovKupRog26} that the $\log \frac{L_l}{L_g}$ term cannot be removed.
We extend this result to the case of time-varying communication networks. 

\section{Preliminaries}\label{sec:prelims}
\subsection{Basic Definitions}\label{sec:defs}
\textbf{Communication graphs.} By $\lambda_i(G)$ we denote the $i$-th eigenvalue of the Laplacian $\cL_G$ of graph $G = (V_G, E_G), |V_G| = n$, so that $\lambda_2(G)$ and $\lambda_n(G)$ are the minimum nonzero and maximum eigenvalues respectively.
$\cA_G$ denotes the adjacency matrix of $G$.
$E_G(S, T)$ denotes the set of edges between $S \subset V_G$ and $T \subset V_G$ in graph $G$.
\\
\textbf{Smoothness and strong convexity.}
A differentiable function $h:\R^d\to \R$ is called $L$-smooth if its gradient is $L$-Lipschitz continuous:
\begin{equation}\label{eq:smooth}
  \norm{\nabla h(x)-\nabla h(y)} \le L \norm{x-y}, \qquad \forall x,y\in \R^d,
\end{equation}
and $\mu$-strongly convex if for all $x,y\in\R^d$,
\begin{equation}\label{eq:strongcvx}
  h(y)\ge h(x) + \langle \nabla h(x), y-x\rangle + \frac{\mu}{2}\sqn{y-x}.
\end{equation}
If $h$ is twice continuously differentiable, equivalent definitions are $L = \sup_x \lambda_{\max}(\nabla^2 h(x))$  and  $\mu = \inf_x \lambda_{\min}(\nabla^2 h(x))$.
We denote by $L_i,\mu_i$ the smoothness and strong convexity constants of $f_i$,
$L_l {=} \max_{i\in[n]}L_i$, $\mu_l {=} \min_{i\in[n]}\mu_i$, and by $L_g,\mu_g$ the corresponding constants of $f$.
If $\mu_g>0$ then the minimizer $x^*$ of \eqref{eq:prob} is unique.
\\
\textbf{Time-varying communication matrices.}
Communication rounds are indexed by $t=1,2,\ldots$ and are associated with a (finite or infinite) sequence of undirected connected graphs
$\cG=\{G^{(t)}\}_{t\ge 1}$ on the same node set of size $n$.
For each communication graph $G^{(t)}$ we are given a compatible communication matrix $W^{(t)}$, i.e.
$W^{(t)}$ is symmetric positive semidefinite; $W^{(t)}_{ij} > 0 \Rightarrow (i, j) \in E({G^{(t)}})$; $\ker W^{(t)} = \operatorname{span}(\{1_n\})$.
For our lower bound we will set $W^{(t)} {=} \cL_{G^{(t)}}$.
We assume uniform spectral bounds for communication matrices: there exist $\underline\lambda>0$ and $\bar\lambda<\infty$ such that for all $t$,
\begin{equation}\label{eq:uniform-spectrum}
  \lminp\!\left(W^{(t)}\right)=\lambda_2\!\left(G^{(t)}\right)\ge \underline\lambda,
  \qquad
  \lmax\!\left(W^{(t)}\right)=\lambda_n\!\left(G^{(t)}\right)\le \bar\lambda.
\end{equation}
We denote the condition number of the sequence of graphs as 
$\chi_\cG = \bar \lambda / \underline \lambda$.
For a fixed graph $G$ we also write $\chi_G {=} \lambda_n(G)/\lambda_2(G)$.
\\
\textbf{$\eps$-solution.}
We say that $\{x^k_i \in \mathcal M_{i, k}\}_{i=1}^n$ is an $\eps$-solution of \eqref{eq:prob} if
\begin{equation}\label{eq:eps-sol}
  \max_{i\in[n]} \norm{x_i^k-x^*}\le \eps.
\end{equation}

\medskip
We introduce the notion of \textit{infection} as an intuitive reference to the information spread. 
In our worst-case problem a distributed algorithm cannot improve approximate solution at any node until information obtained via computation round  in one node reaches another node.
To measure communication complexity we will track the number of communication rounds (or \textit{infection steps}) required to spread the infection by bounding an \textit{effective distance}.

Formally, we define \textit{effective distance} $\rho(S, T)$ between node subsets $S, T$ in a sequence of graphs $\cG$, $s,t \in V(\cG)$ as the minimal number of communication rounds that is enough to spread infection from any $s \in S$ to any $t \in T$.
A similar notion was also introduced  in \cite{metelev2023consensus} under the name of \textit{information flow}.

\subsection{Computation Model}\label{sec:comp-model}
\textbf{Decentralized first-order algorithms.}
We use the standard first-order oracle model \cite{scaman2017optimal}, closely related to the class of zero-respecting algorithms \cite{carmon2020lower}. 
Each node $i\in[n]$ maintains a memory set $\mathcal{M}_{i,k}\subset \R^d$.
Initially, each node has the same starting point, without loss of generality we assume it to be the zero vector: $\mathcal{M}_{i,0}=\{0_d\}$.
At step $k$ a \textit{decentralized first-order algorithm} performs either a local first-order computation at every node,
\begin{equation}\label{eq:local-round}
  \mathcal{M}_{i,k+1}=\operatorname{span}\Big(\mathcal{M}_{i,k}\cup \br{\nabla f_i(x): x\in \mathcal{M}_{i,k}}\Big),
\end{equation}
or a communication round $t$ in which every node updates its memory with its neighbors' memory using communication matrix
\begin{equation}\label{eq:comm-round}
  \mathcal{M}_{i,k+1}=\operatorname{span}\left(\mathcal{M}_{i,k}\cup \br{\sumjn W^{(t)}_{ij}x_j: x_j\in \mathcal{M}_{j,k}}\right).
\end{equation}
The number of communication rounds is denoted by $N_\cG$.

Actually, we will bound communication complexity of more general \textit{distributed first-order algorithms}, which are not restricted by a communication matrix during the communication step:
\begin{equation}\label{eq:comm-round-dist}
  \mathcal{M}_{i,k+1}=\operatorname{span}\left(\mathcal{M}_{i,k}\cup \br{\bigcup_{j : (i,j) \in E(G^{(t)})} \mathcal{M}_{j,k}}\right).
\end{equation}

This model of communication complexity is reminiscent of the round complexity of the LOCAL model, familiar in distributed computing, with one defining distinction. In LOCAL the nodes may exchange arbitrary messages, so within a number of rounds equal to the graph diameter every node can learn the entire topology together with all inputs, and the round complexity of any problem is bounded by $O(n)$. In our model the nodes are restricted to sharing oracle responses and their linear combinations \eqref{eq:comm-round}--\eqref{eq:comm-round-dist}, which cannot reveal the local functions $f_i$ in full. Consequently the communication complexity is no longer bounded by the number of vertices: as our lower bound \eqref{eq:lb_contrib} shows, $N_\cG$ can substantially exceed $n$, since the factors $\sqrt{\kappa_g}$ and $\log\frac1\eps$ grow independently of the network size.

\section{Main Result}

A worst-case communication network for the time-varying setup under the assumption of local $L_l$-smoothness and $\mu_l$-strong convexity is a star graph whose center changes cyclically at each communication round over a subset of intermediate vertices \cite{kovalev2021lower}.
We extend the corresponding lower bound on communication complexity for the nonuniform strongly convex setup (the stricter $\mu_l$-assumption is relaxed to the more general $\mu_g$-assumption) by embedding the star construction into each edge of an expander graph and applying an edge patch procedure using another expander graph on top of this to neutralize the reduction of spectral connectivity.

This approach follows the idea of the lower bound for static graphs \cite{KovKupRog26}, where each edge of an expander graph was replaced with the path of length $\Omega(\sqrt{\chi_G})$.
However, we cannot directly follow this approach for time-varying setup: replacing each edge with a star not only reduces the spectral connectivity $\underline \lambda$ proportionally to the increase in diameter, but also increases maximum degree (related to $\bar \lambda$), leading to the same $\Delta(\cG)= \Omega(\sqrt{\chi_\cG})$ relation as in the static case. 
To achieve the desired $\Delta(\cG)= \Omega(\chi_\cG)$ relation, we develop the edge patch technique which restores spectral connectivity without harming the maximum degree. 

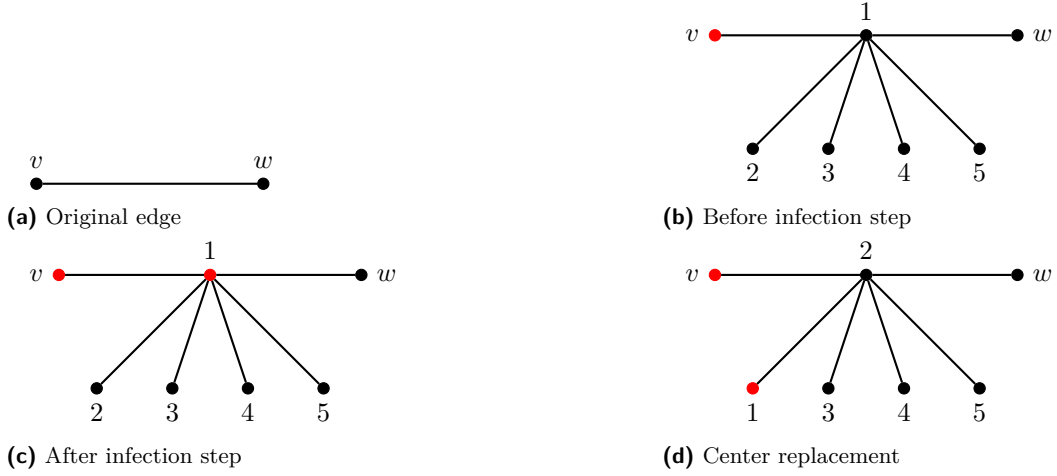
\begin{figure}[h]
    \centering
    \subcaptionbox{Original edge}{
        \begin{tikzpicture}
            \node[circle, draw, fill=black, inner sep=1.5pt, label=above:$v$] (v) at (0,0) {};
            \node[circle, draw, fill=black, inner sep=1.5pt, label=above:$w$] (w) at (3,0) {};
            \draw[thick] (v) -- (w);
        \end{tikzpicture}
    }
    \hfill
    \subcaptionbox{Before infection step}{
        \begin{tikzpicture}
            \node[circle, draw, fill=red, color=red, inner sep=1.5pt, label=left:$v$] (v) at (0,0) {};
            \node[circle, draw, fill=black, inner sep=1.5pt, label=right:$w$] (w) at (4,0) {};
            
            \node[circle, draw, fill=black, inner sep=1.5pt, label=above:$1$] (c) at (2,0) {};
            
            \draw[thick] (v) -- (c) -- (w);
            
            \node[circle, draw, fill=black, inner sep=1.5pt, label=below:$2$] (n2) at (0.5, -1.5) {};
            \draw[thick] (c) -- (n2);
            \node[circle, draw, fill=black, inner sep=1.5pt, label=below:$3$] (n3) at (1.5, -1.5) {};
            \draw[thick] (c) -- (n3);
            \node[circle, draw, fill=black, inner sep=1.5pt, label=below:$4$] (n4) at (2.5, -1.5) {};
            \draw[thick] (c) -- (n4);
            \node[circle, draw, fill=black, inner sep=1.5pt, label=below:$5$] (n5) at (3.5, -1.5) {};
            \draw[thick] (c) -- (n5);
        \end{tikzpicture}
    }
    \hfill
    \subcaptionbox{After infection step}{
        \begin{tikzpicture}
            \node[circle, draw, color=red, fill=red, inner sep=1.5pt, label=left:$v$] (v) at (0,0) {};
            \node[circle, draw, fill=black, inner sep=1.5pt, label=right:$w$] (w) at (4,0) {};
            
            \node[circle, draw, color=red, fill=red, inner sep=1.5pt, label=above:$1$] (c) at (2,0) {};
            
            \draw[thick] (v) -- (c) -- (w);
            
            \node[circle, draw, fill=black, inner sep=1.5pt, label=below:$2$] (n2) at (0.5, -1.5) {};
            \draw[thick] (c) -- (n2);
            \node[circle, draw, fill=black, inner sep=1.5pt, label=below:$3$] (n3) at (1.5, -1.5) {};
            \draw[thick] (c) -- (n3);
            \node[circle, draw, fill=black, inner sep=1.5pt, label=below:$4$] (n4) at (2.5, -1.5) {};
            \draw[thick] (c) -- (n4);
            \node[circle, draw, fill=black, inner sep=1.5pt, label=below:$5$] (n5) at (3.5, -1.5) {};
            \draw[thick] (c) -- (n5);
        \end{tikzpicture}
    }
    \hfill
    \subcaptionbox{Center replacement}{
        \begin{tikzpicture}
            \node[circle, draw, color=red, fill=red, inner sep=1.5pt, label=left:$v$] (v) at (0,0) {};
            \node[circle, draw, fill=black, inner sep=1.5pt, label=right:$w$] (w) at (4,0) {};
            
            \node[circle, draw, fill=black, inner sep=1.5pt, label=above:$2$] (c) at (2,0) {};
            
            \draw[thick] (v) -- (c) -- (w);
            
            \node[circle, draw,  color=red, fill=red, inner sep=1.5pt, label=below:$1$] (n2) at (0.5, -1.5) {};
            \draw[thick] (c) -- (n2);
            \node[circle, draw, fill=black, inner sep=1.5pt, label=below:$3$] (n3) at (1.5, -1.5) {};
            \draw[thick] (c) -- (n3);
            \node[circle, draw, fill=black, inner sep=1.5pt, label=below:$4$] (n4) at (2.5, -1.5) {};
            \draw[thick] (c) -- (n4);
            \node[circle, draw, fill=black, inner sep=1.5pt, label=below:$5$] (n5) at (3.5, -1.5) {};
            \draw[thick] (c) -- (n5);
        \end{tikzpicture}
    }
    \caption{$\tau$-star extension construction for $\tau=5$. (a) Edge $(v, w)$ of the initial graph $H$. (b,c) Infection step in the star subgraph of $G_0$ embedded in edge $(v, w)$. (d) Infected center $1$ is replaced with uninfected leaf $2$ to delay infection of $w$.}
\label{fig:stars}
\end{figure}

\subsection{Graph Construction}
To construct a sequence of communication graphs, we take a $(d = 8)$-regular Margulis expander $H$ on $|V_H| = m^2 \sim \frac{L_l}{L_g}$ vertices \cite[Theorem 8.2]{hoory2006expander}.
Since the condition number of the Laplacian matrix of the graph $H$ does not depend on $m$, in order to obtain a lower bound in terms of $\chi$ we construct, based on the graph $H$, its $\tau$-star extension $G_0$, in which each edge $(v, w)$ is replaced by a star graph consisting of the leaves $v, w$ and $\tau$ additional new vertices.
To slow down the infection process (information spread), when the star center is infected, we replace the center with an uninfected leaf by changing the edges as shown in Figure~\ref{fig:stars}.
Consider vertices $s, t \in V_H$ whose distance in the graph $H$ is equal to its diameter $\Delta(H)$.
This construction guarantees that in order to transmit information from $s$ to $t$ at least $(\tau+1)\Delta(H)$ communication rounds are required. Indeed, to transmit information from vertex $v$ to vertex $w$ such that $(v, w) \in E_H$, two consecutive communications with the same central vertex are needed, which requires $\tau+1$ communication rounds / infection steps.

Denote by $n = |V_{G_0}| = |V_H|(1 + d\tau/2)$ the number of nodes in graph $G_0$ as well as in graphs $G_1$ and $G$, which will be built on the same node set.
Let $U \subset V_G$ be the subset of infected nodes, and define $U_H$ as its restriction on the nodes from $H$. 
Define $\partial_{G_0}U = B\subset \bar U$ the set of nodes that will be infected at the next step (node boundary of $U$ in $G_0$), see \Cref{fig:u}.
\begin{figure}[h!]
    \centering

\tikzset{every picture/.style={line width=0.75pt}} %

\begin{tikzpicture}[x=0.75pt,y=0.75pt,yscale=-1,xscale=1]

\draw    (514.19,462.76) .. controls (533.38,463.22) and (536.18,355.22) .. (514.19,354.33) ;
\draw    (476.48,463.63) .. controls (451.01,463.54) and (451.01,355.94) .. (476.48,355.21) ;
\draw    (549.98,462.9) .. controls (572.24,462.61) and (574.24,355.61) .. (548.41,355.34) ;
\draw    (439.98,355.34) .. controls (418.02,356.14) and (418.02,463.14) .. (440.34,463.63) ;
\draw [color={rgb, 255:red, 255; green, 0; blue, 0 }  ,draw opacity=1 ][line width=0.75]  [dash pattern={on 0.84pt off 2.51pt}]  (549.98,354.47) .. controls (549.24,331.21) and (442.24,330.21) .. (441.55,354.47) ;
\draw [color={rgb, 255:red, 255; green, 0; blue, 0 }  ,draw opacity=1 ][line width=0.75]  [dash pattern={on 0.84pt off 2.51pt}]  (441.55,462.9) .. controls (441.24,489.01) and (550.24,490.01) .. (549.98,462.9) ;
\draw [fill={rgb, 255:red, 63; green, 221; blue, 26 }  ,fill opacity=0.24 ]   (550.34,270) .. controls (547.22,283.47) and (464.11,331.62) .. (466,354.33) .. controls (467.9,377.04) and (490.93,358.97) .. (502.14,378.43) .. controls (513.36,397.89) and (486.65,448.56) .. (502.14,474.81) .. controls (517.64,501.06) and (532.15,495.29) .. (586.48,510.95) ;
\draw [fill={rgb, 255:red, 0; green, 0; blue, 0 }  ,fill opacity=0.17 ]   (417.81,523) .. controls (436.97,476.26) and (490.27,466.18) .. (490.1,402.52) .. controls (489.93,338.87) and (445.4,416.02) .. (429.86,402.52) .. controls (414.32,389.03) and (453.59,367.16) .. (453.95,354.33) .. controls (454.31,341.51) and (403.71,319.53) .. (381.67,306.14) ;
\draw [color={rgb, 255:red, 255; green, 0; blue, 0 }  ,draw opacity=1 ][line width=0.75]  [dash pattern={on 0.84pt off 2.51pt}]  (477.69,426.96) -- (514.19,426.82) ;
\draw [color={rgb, 255:red, 255; green, 0; blue, 0 }  ,draw opacity=1 ][line width=0.75]  [dash pattern={on 0.84pt off 2.51pt}]  (477.34,390.34) -- (513.84,390.21) ;
\draw [color={rgb, 255:red, 255; green, 0; blue, 0 }  ,draw opacity=1 ][line width=0.75]  [dash pattern={on 0.84pt off 2.51pt}]  (439.98,355.34) -- (476.48,355.21) ;
\draw [color={rgb, 255:red, 255; green, 0; blue, 0 }  ,draw opacity=1 ][line width=0.75]  [dash pattern={on 0.84pt off 2.51pt}]  (549.98,426.76) .. controls (586.36,523.72) and (406.85,522.52) .. (441.55,426.76) ;
\draw [fill={rgb, 255:red, 63; green, 221; blue, 26 }  ,fill opacity=0.24 ]   (288.67,270) .. controls (285.55,283.47) and (202.44,331.62) .. (204.33,354.33) .. controls (206.23,377.04) and (229.26,358.97) .. (240.48,378.43) .. controls (251.69,397.89) and (224.98,448.56) .. (240.48,474.81) .. controls (255.97,501.06) and (270.48,495.29) .. (324.81,510.95) ;
\draw [fill={rgb, 255:red, 0; green, 0; blue, 0 }  ,fill opacity=0.17 ]   (156.14,523) .. controls (172.37,483.4) and (213.1,470.12) .. (225.07,428.28) .. controls (227.23,420.73) and (228.46,412.25) .. (228.43,402.52) .. controls (228.26,338.87) and (183.73,416.02) .. (168.19,402.52) .. controls (152.65,389.03) and (191.93,367.16) .. (192.29,354.33) .. controls (192.65,341.51) and (142.04,319.53) .. (120,306.14) ;
\draw [fill={rgb, 255:red, 63; green, 221; blue, 26 }  ,fill opacity=0.24 ]   (538.67,0) .. controls (535.55,13.47) and (452.44,61.62) .. (454.33,84.33) .. controls (456.23,107.04) and (479.26,88.97) .. (490.48,108.43) .. controls (501.69,127.89) and (474.98,178.56) .. (490.48,204.81) .. controls (505.97,231.06) and (520.48,225.29) .. (574.81,240.95) ;
\draw [fill={rgb, 255:red, 0; green, 0; blue, 0 }  ,fill opacity=0.17 ]   (406.14,253) .. controls (425.3,206.26) and (478.6,196.18) .. (478.43,132.52) .. controls (478.26,68.87) and (433.73,146.02) .. (418.19,132.52) .. controls (402.65,119.03) and (441.93,97.16) .. (442.29,84.33) .. controls (442.65,71.51) and (392.04,49.53) .. (370,36.14) ;
\draw   (227.83,100) -- (270,142.17) -- (227.83,184.33) -- (185.67,142.17) -- cycle ;
\draw  [line width=0.75]  (502.52,84.33) -- (538.67,120.48) -- (538.67,156.62) -- (502.52,192.76) -- (466.38,192.76) -- (430.24,156.62) -- (430.24,120.48) -- (466.38,84.33) -- cycle ;
\draw [color={rgb, 255:red, 212; green, 50; blue, 194 }  ,draw opacity=1 ]   (502.52,84.33) -- (502.52,120.48) ;
\draw [shift={(502.52,120.48)}, rotate = 90] [color={rgb, 255:red, 212; green, 50; blue, 194 }  ,draw opacity=1 ][fill={rgb, 255:red, 212; green, 50; blue, 194 }  ,fill opacity=1 ][line width=0.75]      (0, 0) circle [x radius= 3.35, y radius= 3.35]   ;
\draw [shift={(502.52,84.33)}, rotate = 90] [color={rgb, 255:red, 212; green, 50; blue, 194 }  ,draw opacity=1 ][fill={rgb, 255:red, 212; green, 50; blue, 194 }  ,fill opacity=1 ][line width=0.75]      (0, 0) circle [x radius= 3.35, y radius= 3.35]   ;
\draw [color={rgb, 255:red, 212; green, 50; blue, 194 }  ,draw opacity=1 ]   (502.52,84.33) -- (538.67,84.33) ;
\draw [shift={(538.67,84.33)}, rotate = 0] [color={rgb, 255:red, 212; green, 50; blue, 194 }  ,draw opacity=1 ][fill={rgb, 255:red, 212; green, 50; blue, 194 }  ,fill opacity=1 ][line width=0.75]      (0, 0) circle [x radius= 3.35, y radius= 3.35]   ;
\draw [shift={(502.52,84.33)}, rotate = 0] [color={rgb, 255:red, 212; green, 50; blue, 194 }  ,draw opacity=1 ][fill={rgb, 255:red, 212; green, 50; blue, 194 }  ,fill opacity=1 ][line width=0.75]      (0, 0) circle [x radius= 3.35, y radius= 3.35]   ;
\draw [color={rgb, 255:red, 212; green, 50; blue, 194 }  ,draw opacity=1 ]   (466.38,120.48) -- (430.24,120.48) ;
\draw [shift={(430.24,120.48)}, rotate = 180] [color={rgb, 255:red, 212; green, 50; blue, 194 }  ,draw opacity=1 ][fill={rgb, 255:red, 212; green, 50; blue, 194 }  ,fill opacity=1 ][line width=0.75]      (0, 0) circle [x radius= 3.35, y radius= 3.35]   ;
\draw [shift={(466.38,120.48)}, rotate = 180] [color={rgb, 255:red, 212; green, 50; blue, 194 }  ,draw opacity=1 ][fill={rgb, 255:red, 212; green, 50; blue, 194 }  ,fill opacity=1 ][line width=0.75]      (0, 0) circle [x radius= 3.35, y radius= 3.35]   ;
\draw [color={rgb, 255:red, 212; green, 50; blue, 194 }  ,draw opacity=1 ]   (430.24,120.48) -- (430.24,84.33) ;
\draw [shift={(430.24,84.33)}, rotate = 270] [color={rgb, 255:red, 212; green, 50; blue, 194 }  ,draw opacity=1 ][fill={rgb, 255:red, 212; green, 50; blue, 194 }  ,fill opacity=1 ][line width=0.75]      (0, 0) circle [x radius= 3.35, y radius= 3.35]   ;
\draw [shift={(430.24,120.48)}, rotate = 270] [color={rgb, 255:red, 212; green, 50; blue, 194 }  ,draw opacity=1 ][fill={rgb, 255:red, 212; green, 50; blue, 194 }  ,fill opacity=1 ][line width=0.75]      (0, 0) circle [x radius= 3.35, y radius= 3.35]   ;
\draw [color={rgb, 255:red, 212; green, 50; blue, 194 }  ,draw opacity=1 ]   (466.38,192.76) -- (466.38,156.62) ;
\draw [shift={(466.38,156.62)}, rotate = 270] [color={rgb, 255:red, 212; green, 50; blue, 194 }  ,draw opacity=1 ][fill={rgb, 255:red, 212; green, 50; blue, 194 }  ,fill opacity=1 ][line width=0.75]      (0, 0) circle [x radius= 3.35, y radius= 3.35]   ;
\draw [shift={(466.38,192.76)}, rotate = 270] [color={rgb, 255:red, 212; green, 50; blue, 194 }  ,draw opacity=1 ][fill={rgb, 255:red, 212; green, 50; blue, 194 }  ,fill opacity=1 ][line width=0.75]      (0, 0) circle [x radius= 3.35, y radius= 3.35]   ;
\draw [color={rgb, 255:red, 212; green, 50; blue, 194 }  ,draw opacity=1 ]   (466.38,192.76) -- (430.24,192.76) ;
\draw [shift={(430.24,192.76)}, rotate = 180] [color={rgb, 255:red, 212; green, 50; blue, 194 }  ,draw opacity=1 ][fill={rgb, 255:red, 212; green, 50; blue, 194 }  ,fill opacity=1 ][line width=0.75]      (0, 0) circle [x radius= 3.35, y radius= 3.35]   ;
\draw [shift={(466.38,192.76)}, rotate = 180] [color={rgb, 255:red, 212; green, 50; blue, 194 }  ,draw opacity=1 ][fill={rgb, 255:red, 212; green, 50; blue, 194 }  ,fill opacity=1 ][line width=0.75]      (0, 0) circle [x radius= 3.35, y radius= 3.35]   ;
\draw [color={rgb, 255:red, 212; green, 50; blue, 194 }  ,draw opacity=1 ]   (538.67,156.62) -- (538.67,192.76) ;
\draw [shift={(538.67,192.76)}, rotate = 90] [color={rgb, 255:red, 212; green, 50; blue, 194 }  ,draw opacity=1 ][fill={rgb, 255:red, 212; green, 50; blue, 194 }  ,fill opacity=1 ][line width=0.75]      (0, 0) circle [x radius= 3.35, y radius= 3.35]   ;
\draw [shift={(538.67,156.62)}, rotate = 90] [color={rgb, 255:red, 212; green, 50; blue, 194 }  ,draw opacity=1 ][fill={rgb, 255:red, 212; green, 50; blue, 194 }  ,fill opacity=1 ][line width=0.75]      (0, 0) circle [x radius= 3.35, y radius= 3.35]   ;
\draw [color={rgb, 255:red, 212; green, 50; blue, 194 }  ,draw opacity=1 ]   (538.67,156.62) -- (502.52,156.62) ;
\draw [shift={(502.52,156.62)}, rotate = 180] [color={rgb, 255:red, 212; green, 50; blue, 194 }  ,draw opacity=1 ][fill={rgb, 255:red, 212; green, 50; blue, 194 }  ,fill opacity=1 ][line width=0.75]      (0, 0) circle [x radius= 3.35, y radius= 3.35]   ;
\draw [shift={(538.67,156.62)}, rotate = 180] [color={rgb, 255:red, 212; green, 50; blue, 194 }  ,draw opacity=1 ][fill={rgb, 255:red, 212; green, 50; blue, 194 }  ,fill opacity=1 ][line width=0.75]      (0, 0) circle [x radius= 3.35, y radius= 3.35]   ;
\draw    (179.88,354.47) -- (288.31,354.47) ;
\draw    (179.88,390.61) -- (288.31,390.61) ;
\draw    (179.88,462.9) -- (288.31,462.9) ;
\draw    (179.88,462.9) -- (179.88,354.47) ;
\draw    (216.03,462.9) -- (216.03,354.47) ;
\draw    (288.31,462.9) -- (288.31,354.47) ;
\draw    (252.17,462.9) -- (252.17,354.47) ;
\draw [color={rgb, 255:red, 245; green, 166; blue, 35 }  ,draw opacity=1 ]   (227.83,100) ;
\draw [shift={(227.83,100)}, rotate = 0] [color={rgb, 255:red, 245; green, 166; blue, 35 }  ,draw opacity=1 ][fill={rgb, 255:red, 245; green, 166; blue, 35 }  ,fill opacity=1 ][line width=0.75]      (0, 0) circle [x radius= 3.35, y radius= 3.35]   ;
\draw [shift={(227.83,100)}, rotate = 0] [color={rgb, 255:red, 245; green, 166; blue, 35 }  ,draw opacity=1 ][fill={rgb, 255:red, 245; green, 166; blue, 35 }  ,fill opacity=1 ][line width=0.75]      (0, 0) circle [x radius= 3.35, y radius= 3.35]   ;
\draw [color={rgb, 255:red, 0; green, 204; blue, 204 }  ,draw opacity=1 ]   (270,142.17) ;
\draw [shift={(270,142.17)}, rotate = 0] [color={rgb, 255:red, 0; green, 204; blue, 204 }  ,draw opacity=1 ][fill={rgb, 255:red, 0; green, 204; blue, 204 }  ,fill opacity=1 ][line width=0.75]      (0, 0) circle [x radius= 3.35, y radius= 3.35]   ;
\draw [shift={(270,142.17)}, rotate = 0] [color={rgb, 255:red, 0; green, 204; blue, 204 }  ,draw opacity=1 ][fill={rgb, 255:red, 0; green, 204; blue, 204 }  ,fill opacity=1 ][line width=0.75]      (0, 0) circle [x radius= 3.35, y radius= 3.35]   ;
\draw [color={rgb, 255:red, 208; green, 2; blue, 27 }  ,draw opacity=1 ]   (227.83,184.33) ;
\draw [shift={(227.83,184.33)}, rotate = 0] [color={rgb, 255:red, 208; green, 2; blue, 27 }  ,draw opacity=1 ][fill={rgb, 255:red, 208; green, 2; blue, 27 }  ,fill opacity=1 ][line width=0.75]      (0, 0) circle [x radius= 3.35, y radius= 3.35]   ;
\draw [shift={(227.83,184.33)}, rotate = 0] [color={rgb, 255:red, 208; green, 2; blue, 27 }  ,draw opacity=1 ][fill={rgb, 255:red, 208; green, 2; blue, 27 }  ,fill opacity=1 ][line width=0.75]      (0, 0) circle [x radius= 3.35, y radius= 3.35]   ;
\draw    (288.31,390.61) .. controls (324.69,293.61) and (143.97,294.82) .. (179.88,390.61) ;
\draw    (216.03,462.9) .. controls (190.56,462.8) and (190.56,355.2) .. (216.03,354.47) ;
\draw    (252.17,462.9) .. controls (271.36,463.36) and (274.16,355.36) .. (252.17,354.47) ;
\draw [color={rgb, 255:red, 126; green, 211; blue, 33 }  ,draw opacity=1 ]   (185.67,142.17) ;
\draw [shift={(185.67,142.17)}, rotate = 0] [color={rgb, 255:red, 126; green, 211; blue, 33 }  ,draw opacity=1 ][fill={rgb, 255:red, 126; green, 211; blue, 33 }  ,fill opacity=1 ][line width=0.75]      (0, 0) circle [x radius= 3.35, y radius= 3.35]   ;
\draw [shift={(185.67,142.17)}, rotate = 0] [color={rgb, 255:red, 126; green, 211; blue, 33 }  ,draw opacity=1 ][fill={rgb, 255:red, 126; green, 211; blue, 33 }  ,fill opacity=1 ][line width=0.75]      (0, 0) circle [x radius= 3.35, y radius= 3.35]   ;
\draw [color={rgb, 255:red, 126; green, 211; blue, 33 }  ,draw opacity=1 ]   (430.24,156.62) ;
\draw [shift={(430.24,156.62)}, rotate = 0] [color={rgb, 255:red, 126; green, 211; blue, 33 }  ,draw opacity=1 ][fill={rgb, 255:red, 126; green, 211; blue, 33 }  ,fill opacity=1 ][line width=0.75]      (0, 0) circle [x radius= 3.35, y radius= 3.35]   ;
\draw [shift={(430.24,156.62)}, rotate = 0] [color={rgb, 255:red, 126; green, 211; blue, 33 }  ,draw opacity=1 ][fill={rgb, 255:red, 126; green, 211; blue, 33 }  ,fill opacity=1 ][line width=0.75]      (0, 0) circle [x radius= 3.35, y radius= 3.35]   ;
\draw [color={rgb, 255:red, 245; green, 166; blue, 35 }  ,draw opacity=1 ]   (466.38,84.33) ;
\draw [shift={(466.38,84.33)}, rotate = 0] [color={rgb, 255:red, 245; green, 166; blue, 35 }  ,draw opacity=1 ][fill={rgb, 255:red, 245; green, 166; blue, 35 }  ,fill opacity=1 ][line width=0.75]      (0, 0) circle [x radius= 3.35, y radius= 3.35]   ;
\draw [shift={(466.38,84.33)}, rotate = 0] [color={rgb, 255:red, 245; green, 166; blue, 35 }  ,draw opacity=1 ][fill={rgb, 255:red, 245; green, 166; blue, 35 }  ,fill opacity=1 ][line width=0.75]      (0, 0) circle [x radius= 3.35, y radius= 3.35]   ;
\draw [color={rgb, 255:red, 0; green, 204; blue, 204 }  ,draw opacity=1 ]   (538.67,120.48) ;
\draw [shift={(538.67,120.48)}, rotate = 0] [color={rgb, 255:red, 0; green, 204; blue, 204 }  ,draw opacity=1 ][fill={rgb, 255:red, 0; green, 204; blue, 204 }  ,fill opacity=1 ][line width=0.75]      (0, 0) circle [x radius= 3.35, y radius= 3.35]   ;
\draw [shift={(538.67,120.48)}, rotate = 0] [color={rgb, 255:red, 0; green, 204; blue, 204 }  ,draw opacity=1 ][fill={rgb, 255:red, 0; green, 204; blue, 204 }  ,fill opacity=1 ][line width=0.75]      (0, 0) circle [x radius= 3.35, y radius= 3.35]   ;
\draw [color={rgb, 255:red, 208; green, 2; blue, 27 }  ,draw opacity=1 ]   (502.52,192.76) ;
\draw [shift={(502.52,192.76)}, rotate = 0] [color={rgb, 255:red, 208; green, 2; blue, 27 }  ,draw opacity=1 ][fill={rgb, 255:red, 208; green, 2; blue, 27 }  ,fill opacity=1 ][line width=0.75]      (0, 0) circle [x radius= 3.35, y radius= 3.35]   ;
\draw [shift={(502.52,192.76)}, rotate = 0] [color={rgb, 255:red, 208; green, 2; blue, 27 }  ,draw opacity=1 ][fill={rgb, 255:red, 208; green, 2; blue, 27 }  ,fill opacity=1 ][line width=0.75]      (0, 0) circle [x radius= 3.35, y radius= 3.35]   ;
\draw [color={rgb, 255:red, 0; green, 204; blue, 204 }  ,draw opacity=1 ]   (288.31,390.61) ;
\draw [shift={(288.31,390.61)}, rotate = 0] [color={rgb, 255:red, 0; green, 204; blue, 204 }  ,draw opacity=1 ][fill={rgb, 255:red, 0; green, 204; blue, 204 }  ,fill opacity=1 ][line width=0.75]      (0, 0) circle [x radius= 3.35, y radius= 3.35]   ;
\draw [shift={(288.31,390.61)}, rotate = 0] [color={rgb, 255:red, 0; green, 204; blue, 204 }  ,draw opacity=1 ][fill={rgb, 255:red, 0; green, 204; blue, 204 }  ,fill opacity=1 ][line width=0.75]      (0, 0) circle [x radius= 3.35, y radius= 3.35]   ;
\draw [color={rgb, 255:red, 208; green, 2; blue, 27 }  ,draw opacity=1 ]   (252.17,462.9) -- (252.52,462.76) ;
\draw [shift={(252.52,462.76)}, rotate = 338.96] [color={rgb, 255:red, 208; green, 2; blue, 27 }  ,draw opacity=1 ][fill={rgb, 255:red, 208; green, 2; blue, 27 }  ,fill opacity=1 ][line width=0.75]      (0, 0) circle [x radius= 3.35, y radius= 3.35]   ;
\draw [shift={(252.17,462.9)}, rotate = 338.96] [color={rgb, 255:red, 208; green, 2; blue, 27 }  ,draw opacity=1 ][fill={rgb, 255:red, 208; green, 2; blue, 27 }  ,fill opacity=1 ][line width=0.75]      (0, 0) circle [x radius= 3.35, y radius= 3.35]   ;
\draw    (477.69,354.47) -- (549.98,354.47) ;
\draw    (441.55,462.9) -- (549.98,462.9) ;
\draw    (441.55,462.9) -- (441.55,354.47) ;
\draw    (477.69,462.9) -- (477.69,354.47) ;
\draw    (549.98,462.9) -- (549.98,354.47) ;
\draw    (513.84,462.9) -- (513.84,390.21) -- (513.84,354.47) ;
\draw    (549.98,390.61) .. controls (586.36,293.61) and (405.64,294.82) .. (441.55,390.61) ;
\draw [color={rgb, 255:red, 144; green, 19; blue, 254 }  ,draw opacity=1 ][line width=1.5]    (549.98,426.76) .. controls (477.01,462.35) and (550.01,426.85) .. (477.69,462.9) ;
\draw [color={rgb, 255:red, 208; green, 2; blue, 27 }  ,draw opacity=1 ]   (513.84,462.9) -- (514.19,462.76) ;
\draw [shift={(514.19,462.76)}, rotate = 338.96] [color={rgb, 255:red, 208; green, 2; blue, 27 }  ,draw opacity=1 ][fill={rgb, 255:red, 208; green, 2; blue, 27 }  ,fill opacity=1 ][line width=0.75]      (0, 0) circle [x radius= 3.35, y radius= 3.35]   ;
\draw [shift={(513.84,462.9)}, rotate = 338.96] [color={rgb, 255:red, 208; green, 2; blue, 27 }  ,draw opacity=1 ][fill={rgb, 255:red, 208; green, 2; blue, 27 }  ,fill opacity=1 ][line width=0.75]      (0, 0) circle [x radius= 3.35, y radius= 3.35]   ;
\draw    (179.88,426.76) -- (216.38,426.62) ;
\draw    (251.82,426.89) -- (288.31,426.76) ;
\draw [color={rgb, 255:red, 255; green, 0; blue, 0 }  ,draw opacity=1 ][line width=1.5]    (216.38,426.62) -- (252.88,426.48) ;
\draw [color={rgb, 255:red, 255; green, 0; blue, 0 }  ,draw opacity=1 ][line width=1.5]    (215.85,390.68) -- (252.35,390.54) ;
\draw [color={rgb, 255:red, 255; green, 0; blue, 0 }  ,draw opacity=1 ][line width=1.5]    (179.53,354.61) -- (216.03,354.47) ;
\draw   (430.24,111.48) -- (433.18,116.84) -- (439.75,117.7) -- (434.99,121.87) -- (436.12,127.76) -- (430.24,124.98) -- (424.36,127.76) -- (425.48,121.87) -- (420.73,117.7) -- (427.3,116.84) -- cycle ;
\draw   (466.38,183.76) -- (469.32,189.12) -- (475.89,189.98) -- (471.14,194.15) -- (472.26,200.04) -- (466.38,197.26) -- (460.5,200.04) -- (461.63,194.15) -- (456.87,189.98) -- (463.44,189.12) -- cycle ;
\draw   (216.03,453.9) -- (218.97,459.26) -- (225.54,460.12) -- (220.78,464.29) -- (221.9,470.18) -- (216.03,467.4) -- (210.15,470.18) -- (211.27,464.29) -- (206.52,460.12) -- (213.09,459.26) -- cycle ;
\draw   (179.88,381.61) -- (182.82,386.97) -- (189.39,387.83) -- (184.64,392) -- (185.76,397.89) -- (179.88,395.11) -- (174.01,397.89) -- (175.13,392) -- (170.37,387.83) -- (176.95,386.97) -- cycle ;
\draw    (513.38,390.61) -- (549.98,390.61) ;
\draw    (441.55,390.61) -- (477.38,390.61) -- (477.38,390.61) -- cycle ;
\draw [color={rgb, 255:red, 144; green, 19; blue, 254 }  ,draw opacity=1 ][line width=1.5]    (477.69,462.9) .. controls (440.71,426.35) and (477.21,462.85) .. (441.55,426.76) ;
\draw    (513.84,426.55) -- (514.38,426.76) -- (549.98,426.76) ;
\draw    (441.55,426.76) -- (442.1,426.96) -- (477.69,426.96) ;
\draw [color={rgb, 255:red, 144; green, 19; blue, 254 }  ,draw opacity=1 ][line width=1.5]    (441.55,390.61) .. controls (430.38,390.21) and (431.38,355.21) .. (441.55,354.47) ;
\draw [color={rgb, 255:red, 144; green, 19; blue, 254 }  ,draw opacity=1 ][line width=1.5]    (441.55,390.61) .. controls (461.38,370.21) and (463.38,370.21) .. (477.69,354.47) ;
\draw [color={rgb, 255:red, 144; green, 19; blue, 254 }  ,draw opacity=1 ][line width=1.5]    (477.69,462.9) .. controls (487.38,462.21) and (487.38,426.21) .. (477.69,426.96) ;
\draw [color={rgb, 255:red, 144; green, 19; blue, 254 }  ,draw opacity=1 ][line width=1.5]    (477.69,462.9) -- (513.84,426.55) ;
\draw [color={rgb, 255:red, 144; green, 19; blue, 254 }  ,draw opacity=1 ][line width=1.5]    (477.38,390.61) .. controls (477.38,404.21) and (442.38,402.21) .. (442.1,390.82) ;
\draw [color={rgb, 255:red, 144; green, 19; blue, 254 }  ,draw opacity=1 ][line width=1.5]    (513.38,390.61) .. controls (513.38,410.21) and (441.38,412.21) .. (441.55,390.61) ;
\draw [color={rgb, 255:red, 245; green, 166; blue, 35 }  ,draw opacity=1 ]   (216.03,354.47) ;
\draw [shift={(216.03,354.47)}, rotate = 0] [color={rgb, 255:red, 245; green, 166; blue, 35 }  ,draw opacity=1 ][fill={rgb, 255:red, 245; green, 166; blue, 35 }  ,fill opacity=1 ][line width=0.75]      (0, 0) circle [x radius= 3.35, y radius= 3.35]   ;
\draw [shift={(216.03,354.47)}, rotate = 0] [color={rgb, 255:red, 245; green, 166; blue, 35 }  ,draw opacity=1 ][fill={rgb, 255:red, 245; green, 166; blue, 35 }  ,fill opacity=1 ][line width=0.75]      (0, 0) circle [x radius= 3.35, y radius= 3.35]   ;
\draw [color={rgb, 255:red, 245; green, 166; blue, 35 }  ,draw opacity=1 ]   (477.69,354.47) ;
\draw [shift={(477.69,354.47)}, rotate = 0] [color={rgb, 255:red, 245; green, 166; blue, 35 }  ,draw opacity=1 ][fill={rgb, 255:red, 245; green, 166; blue, 35 }  ,fill opacity=1 ][line width=0.75]      (0, 0) circle [x radius= 3.35, y radius= 3.35]   ;
\draw [shift={(477.69,354.47)}, rotate = 0] [color={rgb, 255:red, 245; green, 166; blue, 35 }  ,draw opacity=1 ][fill={rgb, 255:red, 245; green, 166; blue, 35 }  ,fill opacity=1 ][line width=0.75]      (0, 0) circle [x radius= 3.35, y radius= 3.35]   ;
\draw [color={rgb, 255:red, 126; green, 211; blue, 33 }  ,draw opacity=1 ]   (441.55,426.76) ;
\draw [shift={(441.55,426.76)}, rotate = 0] [color={rgb, 255:red, 126; green, 211; blue, 33 }  ,draw opacity=1 ][fill={rgb, 255:red, 126; green, 211; blue, 33 }  ,fill opacity=1 ][line width=0.75]      (0, 0) circle [x radius= 3.35, y radius= 3.35]   ;
\draw [shift={(441.55,426.76)}, rotate = 0] [color={rgb, 255:red, 126; green, 211; blue, 33 }  ,draw opacity=1 ][fill={rgb, 255:red, 126; green, 211; blue, 33 }  ,fill opacity=1 ][line width=0.75]      (0, 0) circle [x radius= 3.35, y radius= 3.35]   ;
\draw [color={rgb, 255:red, 0; green, 204; blue, 204 }  ,draw opacity=1 ]   (549.98,390.61) ;
\draw [shift={(549.98,390.61)}, rotate = 0] [color={rgb, 255:red, 0; green, 204; blue, 204 }  ,draw opacity=1 ][fill={rgb, 255:red, 0; green, 204; blue, 204 }  ,fill opacity=1 ][line width=0.75]      (0, 0) circle [x radius= 3.35, y radius= 3.35]   ;
\draw [shift={(549.98,390.61)}, rotate = 0] [color={rgb, 255:red, 0; green, 204; blue, 204 }  ,draw opacity=1 ][fill={rgb, 255:red, 0; green, 204; blue, 204 }  ,fill opacity=1 ][line width=0.75]      (0, 0) circle [x radius= 3.35, y radius= 3.35]   ;
\draw    (179.53,354.61) .. controls (157.57,355.41) and (157.57,462.41) .. (179.88,462.9) ;
\draw [color={rgb, 255:red, 255; green, 0; blue, 0 }  ,draw opacity=1 ][line width=1.5]    (288.31,354.47) .. controls (287.57,331.21) and (180.57,330.21) .. (179.88,354.47) ;
\draw    (288.31,462.9) .. controls (310.57,462.61) and (312.57,355.61) .. (286.74,355.34) ;
\draw [color={rgb, 255:red, 255; green, 0; blue, 0 }  ,draw opacity=1 ][line width=1.5]    (179.88,462.9) .. controls (179.57,489.01) and (288.57,490.01) .. (288.31,462.9) ;
\draw [color={rgb, 255:red, 144; green, 19; blue, 254 }  ,draw opacity=1 ][line width=1.5]    (477.69,462.9) .. controls (477.38,475.01) and (440.38,472.01) .. (441.55,462.9) ;
\draw [color={rgb, 255:red, 144; green, 19; blue, 254 }  ,draw opacity=1 ][line width=1.5]    (549.98,462.9) .. controls (549.66,475.01) and (478.52,479.01) .. (477.69,462.9) ;
\draw [color={rgb, 255:red, 144; green, 19; blue, 254 }  ,draw opacity=1 ][line width=1.5]    (439.98,355.34) .. controls (419.38,356.01) and (419.38,391.01) .. (441.55,390.61) ;
\draw [color={rgb, 255:red, 144; green, 19; blue, 254 }  ,draw opacity=1 ][line width=1.5]    (442.1,390.82) .. controls (549.38,354.01) and (441.38,391.01) .. (549.98,354.47) ;
\draw   (441.55,381.61) -- (444.49,386.97) -- (451.06,387.83) -- (446.31,392) -- (447.43,397.89) -- (441.55,395.11) -- (435.67,397.89) -- (436.8,392) -- (432.04,387.83) -- (438.61,386.97) -- cycle ;
\draw   (477.69,453.9) -- (480.63,459.26) -- (487.21,460.12) -- (482.45,464.29) -- (483.57,470.18) -- (477.69,467.4) -- (471.82,470.18) -- (472.94,464.29) -- (468.18,460.12) -- (474.76,459.26) -- cycle ;
\draw [color={rgb, 255:red, 255; green, 0; blue, 0 }  ,draw opacity=1 ][line width=1.5]    (288.31,426.76) .. controls (324.69,523.72) and (145.18,522.52) .. (179.88,426.76) ;
\draw [color={rgb, 255:red, 126; green, 211; blue, 33 }  ,draw opacity=1 ]   (179.88,426.76) ;
\draw [shift={(179.88,426.76)}, rotate = 0] [color={rgb, 255:red, 126; green, 211; blue, 33 }  ,draw opacity=1 ][fill={rgb, 255:red, 126; green, 211; blue, 33 }  ,fill opacity=1 ][line width=0.75]      (0, 0) circle [x radius= 3.35, y radius= 3.35]   ;
\draw [shift={(179.88,426.76)}, rotate = 0] [color={rgb, 255:red, 126; green, 211; blue, 33 }  ,draw opacity=1 ][fill={rgb, 255:red, 126; green, 211; blue, 33 }  ,fill opacity=1 ][line width=0.75]      (0, 0) circle [x radius= 3.35, y radius= 3.35]   ;

\draw (221,232.4) node [anchor=north west][inner sep=0.75pt]    {$H$};
\draw (472.81,232.4) node [anchor=north west][inner sep=0.75pt]    {$G_{0}$};
\draw (224.58,515.4) node [anchor=north west][inner sep=0.75pt]    {$F$};
\draw (392.53,67.61) node [anchor=north west][inner sep=0.75pt]    {$U$};
\draw (435.53,27.61) node [anchor=north west][inner sep=0.75pt]    {$B$};
\draw (504.53,31.61) node [anchor=north west][inner sep=0.75pt]    {$\overline{U} \backslash B$};
\draw (437.48,517.61) node [anchor=north west][inner sep=0.75pt]    {$F$};
\draw (141.67,337.61) node [anchor=north west][inner sep=0.75pt]    {$U$};
\draw (188.67,297.61) node [anchor=north west][inner sep=0.75pt]    {$B$};
\draw (452.48,517.21) node [anchor=north west][inner sep=0.75pt]   [align=left] {after patch};
\draw (253.67,301.61) node [anchor=north west][inner sep=0.75pt]    {$\overline{U} \backslash B$};
\draw (400.48,337.61) node [anchor=north west][inner sep=0.75pt]    {$U$};
\draw (447.48,297.61) node [anchor=north west][inner sep=0.75pt]    {$B$};
\draw (516.48,301.61) node [anchor=north west][inner sep=0.75pt]    {$\overline{U} \backslash B$};

\end{tikzpicture}
\caption{Stages of constructing the communication graph. The figure illustrates: the initial graph $H$ ($C_4$ in our example); its extension $G_0$ with the highlighted information spread boundary $B = \partial_{G_0} U$ (boundary vertices are denoted by \ding{73}); the overlaid multigraph $F$ (here, a torus graph) containing bad edges (highlighted in {\color[HTML]{FF0000}red}); and the final state of graph $F$ after applying the patch procedure to prevent premature information spread (new edges added during the patch procedure are marked in {\color[HTML]{9013FE}violet}).}
\label{fig:u}
\end{figure}
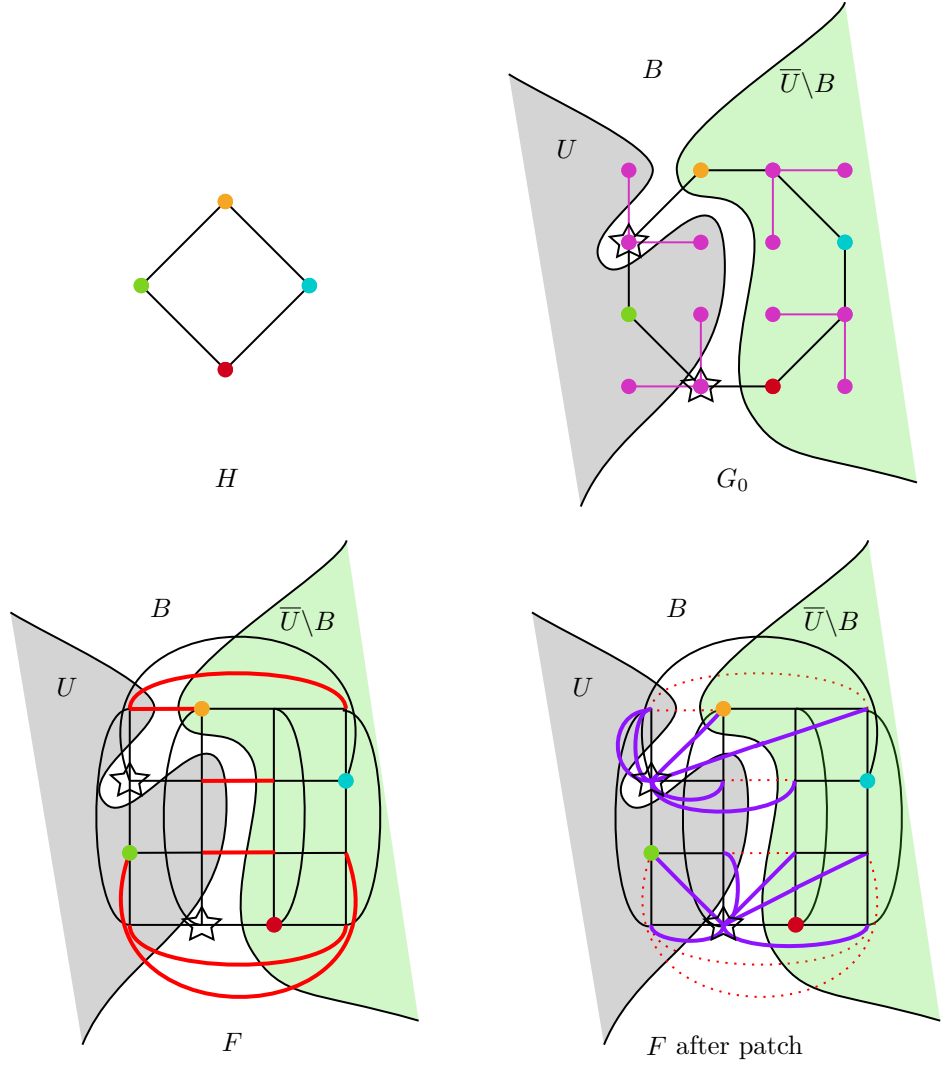

\subsubsection{Patch Procedure}
The patch procedure $G_0 \to G$ adds edges to $G_0$ to increase its spectral connectivity without supporting infection spread. 
  Let $F$ be an expander on $V_{G_0}$ with maximum degree $d_F$.
  Define $G_1 = (V_{G_0}, E_{G_0} \cup E_F)$.
  An edge $(v, w) \in E_F$ is called \textit{bad} if it extends node boundary, i.e. $v \in U$, $w \in \bar U \setminus B$.
  Construct $G$ by replacing each bad edge $(v, w)$ with two edges $(v, b), (b, w)$, where nodes $b \in B$ are assigned to bad edges uniformly, e.g., in a cyclic order.

\begin{claim}\label{claim:cut}
   Any (not necessarily uniform) assignment of nodes $b$ to bad edges does not decrease the size of any edge cut:  $|E_G(A, \bar A)| \geq |E_{G_1}(A, \bar A)|$ $\forall A \subset V_G$.
\end{claim}
\begin{proof}
  $E_{G_1}(A, \bar A) \setminus E_G(A, \bar A)$ includes only bad edges $(v, w):$ $v \in U$, $w \in \bar U \setminus B$.
  Since each bad edge $(v, w) \in E_{G_1}$ is replaced with two new edges $(v, b), (b, w) \in G$ (parallel edges allowed), the loss of $(v,w)$ in  $E_{G_1}(A, \bar A)$ is either compensated by $(v, b)$ if $b  \in \bar A$, or by $(b, w)$ otherwise.
\end{proof}

\begin{claim}\label{claim:maxdeg}
    $\max \deg(G) \leq d + \frac{3}{2}d_F + (d\tau+ 2)\frac{d_F}{\lambda_2(H)} + 1$.
\end{claim}
\begin{proof}
  We begin with a lower bound on the number of nodes in the boundary $B$. By the isoperimetric inequality \cite[Theorem~20.1.1]{spielman2012spectral} applied to the smaller side of the cut $(U_H, \bar U_H)$,
  \begin{equation}\label{eq:B_lb}
  |B| = |E_H(U_H, \bar U_H)| \geq \frac{\lambda_2(H)}{2} \min\{|U_H|, |\bar U_H|\}.
  \end{equation}
  Then, denote $M$ the number of bad edges. It is upper bounded as
  \begin{equation}\label{eq:M_ub}
  \begin{aligned}
        M &= \frac{|E_F(U, \bar U)| + |E_F(U \cup B, \bar U \setminus B)|-|E_F(B, \bar B)|}{2} 
        \\ &\leq
        \frac{|E_F(U, \bar U)| + |E_F(U \cup B, \bar U \setminus B)|}{2}
        \\ &\leq  
        \frac{d_F}{2}(\min\{|U|, |\bar U|\} + \min\{|U \cup B|, |\bar U \setminus B|\}) 
        \\ &= \frac{d_F}{2}(2 \min\{|U|, |\bar U \setminus B|\} + |B|)
  \end{aligned}
  \end{equation}
  Since the assignment of boundary nodes $b \in B$ to bad edges is uniform, the replacement of bad edges increases the degree of each node in $B$ at most by
  \begin{equation}
  \begin{aligned}
  \left\lceil \frac{M}{|B|}\right\rceil 
  &\sref{eq:M_ub}\leq \left\lceil\frac{d_F(2 \min\{|U|, |\bar U \setminus B|\} + |B|)}{2|B|} \right\rceil
  \\&\leq \frac{d_F \min\{|U|, |\bar U|\}}{|B|} + \frac{d_F}{2} + 1  
  \\& \sref{eq:B_lb} \leq
  \frac{2 d_F \min\{|U|, |\bar U|\}}{\lambda_2(H)\min\{|U_H|,|\bar U_H|\}} + \frac{d_F}{2} + 1
  \\&\leq 2 \left(\frac{d \tau}{2}+1\right)\frac{d_F}{\lambda_2(H)} + \frac{d_F}{2} + 1.
  \end{aligned}
  \end{equation}
\end{proof}

We summarize the properties of $G$ in the following lemma
\begin{lemma}\label{lem:patch}
  By adding edges to $G_0$ we can obtain graph $G$ with $\lambda_2(G) \geq \lambda_2(F)/2$, $\lambda_n(G) \leq 2(d + \frac{3}{2} d_F + (d\tau+ 2)\frac{d_F}{\lambda_2(H)} + 1)$ without extending node boundary of $U$: $\partial_{G_0} U = \partial_{G} U = B$.
\end{lemma}
\begin{proof}
    The node boundary $B$ is preserved by the bad edges replacement procedure.
    Each bad edge $(v,w) \in E(F)$ is replaced by two edges $(v,b)$ and $(b,w)$. For any vector $x \in \mathbb{R}^n$, the inequality $(x_v - x_b)^2 + (x_b - x_w)^2 \geq \frac{1}{2}(x_v - x_w)^2$ holds. Thus, the Laplacian of the replaced edges dominates half the Laplacian of the original edge, giving $L_{G} \succeq \frac{1}{2} L_F$. Consequently:
\begin{equation}
    \lambda_2(G^{(k)}) \geq \frac{\lambda_2(F)}{2}
\end{equation}
    The bound for $\lambda_n$ follows from \Cref{claim:maxdeg} and the fact that the maximum eigenvalue of a Laplacian is not greater than the maximum degree multiplied by $2$.
\end{proof}

\begin{claim}\label{clm:margulis}
  Let $H$ be the $(d=8)$-regular Margulis expander on $m^2$ nodes. 
  Then
  \begin{equation}
    d - 5\sqrt 2 \leq \lambda_i(H) \leq 2d, \quad i \in \{2, \ldots, m^2\},
  \end{equation}
  and its diameter
  $\Delta(H) > \log m - 0.15$.
\end{claim}
\begin{proof}
  It holds for any $m \geq 1$ that $\max_{i \in \br{2,\ldots,m^2}}|\lambda_i(\cA_H)| \leq 5 \sqrt{2}$, where $\lambda_i(\cA)$ denote eigenvalues of adjacency matrix ordered from the largest to the smallest \cite[Theorem~8.2]{hoory2006expander}.
  Since $H$ is $d$-regular, we have $\cL_H = dI_{m^2} - \cA_H$, therefore
  \begin{equation}
    \lminp(H) \geq d - 5 \sqrt{2}.
  \end{equation}
  Similarly, in each row $a_i^\top$ of adjacency matrix $\cA_H$ all elements are nonnegative and their sum is equal to $d$, thus its eigenvalue moduli are upper-bounded by $d$: let $x$ be any eigenvector, then
  \begin{equation}
  |\lambda| |x_i| = |a_i^\top x| \leq \sum_{j=1}^{m^2} a_{ij}|x_j| \leq \max_{j \in [m^2]} |x_j| \sum_{j=1}^{m^2} a_{ij} = d\max_{j \in [m^2]} |x_j|.
  \end{equation}
  In turn,
  \begin{equation}
    \lmax(H) \leq d + \max_i |\lambda_i(\cA_H)|\leq 2 d.
  \end{equation}

  Counting the number of vertices at distance at most $\Delta(H)$ from any fixed vertex (including itself), we have the Moore bound
  \begin{equation*}
  \begin{aligned}
    m^2 &\leq 1 + d + d(d-1) + d(d-1)^2 + \ldots + d (d-1)^{\Delta(H)-1} 
    \\&= 1 + d \frac{(d-1)^{\Delta(H)}-1}{d-2}\leq \frac{d}{d-2}(d-1)^{\Delta(H)},
  \end{aligned}
  \end{equation*}
  \begin{equation*}
    \log{m^2} \leq \log{\frac{d}{d-2}} + \Delta(H)\log\cbr{d-1},
  \end{equation*}
  \begin{equation}\label{eq:diam_H_n}
    \Delta(H) \geq \frac{\log{m^2}}{\log\cbr{d-1}} + \frac{\log{\frac{d-2}{d}}}{\log\cbr{d-1}}  \stackrel{d = 8}{>} \frac12 \log m^2 - 0.15 = \log m - 0.15.
  \end{equation}
\end{proof}

\begin{lemma}\label{lem:vertex_merge}
Let $M$ be an undirected $d$-regular graph on $n'$ vertices, and let $F$ be a multigraph on $n>5$ vertices ($n < n'$) obtained by merging $k = n' - n$ disjoint pairs of vertices in $M$. Assume that all edges of $M$ are preserved in $F$ (edges between merged vertices become self-loops, and edges to the same neighbor become parallel edges). Then the second eigenvalue of the Laplacian $L_F$ is bounded as:
\begin{equation}
    \lambda_2(F) \geq \lambda_2(M)
\end{equation}
\end{lemma}

\begin{proof}
Let $L_M$ and $L_F$ be the Laplacians corresponding to graphs $M$ and $F$. By the Courant-Fischer theorem, the second eigenvalue of $F$ is the minimum of the Rayleigh quotient over zero-mean vectors:
\begin{equation}
    \lambda_2(F) = \min_{x \perp \mathbf{1}, x \neq 0} \frac{x^\top L_F x}{\|x\|^2}.
\end{equation}
Let $x \in \mathbb{R}^n$ be the eigenvector achieving this minimum, so $\sum_{v \in V(F)} x_v = 0$ and $x^\top L_F x = \lambda_2(F) \|x\|^2$.
We define a lifted vector $\tilde{x} \in \mathbb{R}^{n'}$ for the graph $M$ by assigning to each vertex in $M$ the exact value of its corresponding merged vertex in $F$. Because the edge sets perfectly correspond (with internal merged edges becoming loops and adjacent edges becoming parallel), the quadratic forms are exactly equal: $x^\top L_F x = \tilde{x}^\top L_M \tilde{x}$.

Let $\mu = \frac{1}{n'} \sum_{u \in V(M)} \tilde{x}_u$ be the mean of $\tilde{x}$.
Since $x$ sums to zero, the sum of $\tilde{x}$ comes exclusively from the $k$ duplicated vertices. Let $x_{m_1}, \ldots, x_{m_k}$ be the values on these vertices. Thus $\mu = \frac{1}{n'} \sum_{i=1}^k x_{m_i}$. Then vector $\tilde{x} - \mu \mathbf{1}$ is orthogonal to $\mathbf{1}$. From the properties of Laplacian and Courant-Fischer theorem:
\begin{equation}
    \tilde{x}^\top L_M \tilde{x} = (\tilde{x} - \mu \mathbf{1})^\top L_M (\tilde{x} - \mu \mathbf{1}) \ge \lambda_2(M) \|\tilde{x} - \mu\mathbf{1}\|^2.
\end{equation}

Expanding the norm of $\tilde{x}$:
\begin{equation}
    \|\tilde{x} - \mu\mathbf{1}\|^2 = \|\tilde{x}\|^2 - n'\mu^2 = \|x\|^2 + \sum_{i=1}^k x_{m_i}^2 - n'\left(\frac{1}{n'}\sum_{i=1}^k x_{m_i}\right)^2.
\end{equation}
Applying the Cauchy-Schwarz inequality yields $\frac{1}{n'}(\sum_{i=1}^k x_{m_i})^2 \leq \frac{k}{n'}\sum_{i=1}^k x_{m_i}^2$. Since $k < n'$, it follows that $n'\mu^2 \leq \sum_{i=1}^k x_{m_i}^2$. 
Therefore, $\|\tilde{x} - \mu\mathbf{1}\|^2 \geq \|x\|^2$. 

Combining these relations yields:
\begin{equation}
    \lambda_2(F)\|x\|^2 = x^\top L_F x = \tilde{x}^\top L_M \tilde{x} \geq \lambda_2(M) \|\tilde{x} - \mu\mathbf{1}\|^2 \geq \lambda_2(M)\|x\|^2.
\end{equation}
Dividing by $\|x\|^2$ establishes $\lambda_2(F) \geq \lambda_2(M)$ which concludes the proof.
\end{proof}

\begin{lemma}\label{lem:distance}
Let $\tau \geq 1, m \geq 2$ be integer parameters. 
For any $n = m^2(1 + 4\tau)$, $n > 5$ there exists a sequence of graphs $\cG = \{G^{(k)}\}_{k=1}^{\Delta(\cG)}$ on $|V_\cG| = n$ nodes, and two sets of nodes $S, T \subset V_\cG$  with effective distance between them bounded as
\begin{equation}
    \rho\left(S, T\right) \geq (\tau+1)(\log m - 2.15)
\end{equation}
and the condition number of the sequence is $\chi_\cG \le 291 + 595 \tau$.
\end{lemma}

\begin{proof}
The sequence $\cG$ is constructed as follows. 
We start by setting $H$ to be the $(d = 8)$-regular Margulis expander on $|V_H| = m^2$ nodes \cite[Theorem 8.2]{hoory2006expander}.
We take $s$ and $t$~--- two nodes in $V_H$ such that the distance between them is the diameter $\Delta(H)$ of graph $H$. 
At the initial moment the communication graph $G_0^{(1)} = G^{(1)}$ is just the $\tau$-star extension of $H$, and the set of infected nodes consists of the single node: $U =\{s\}$.
Note that the condition $m\geq 2$ was required only for $H$ to have more than one node, so that the effective diameter could be increased by the star extension.
After the $k$-th timestep (communication round) we change centers of each star in $G_0^{(k)}$ in cyclic order, obtaining $G_0^{(k+1)}$, and then we apply the patch procedure, obtaining $G^{(k+1)}$.
By \Cref{lem:patch} the patch procedure does not affect information (infection) spread.
Due to the cyclic change of star centers for each $(v, w) \in E_H$ it takes $\tau + 1$ timesteps to transmit infection from $v$ to $w$, thus $t$ becomes infected only after $(\tau + 1) \Delta(H)$ timesteps, yielding
\begin{equation}\label{eq:diams-G-H}
   \rho\left(\{s\}, \{t\}\right) \geq (\tau + 1) \Delta(H). 
\end{equation}

By \Cref{clm:margulis} Margulis expander $H$ has $\Delta(H) > \log m - 0.15$, and $\lambda_2(H) \geq d - 5 \sqrt{2} > 0.928$ regardless of $m$.

For an integer $n>5$, consider a Margulis expander $M$ such that $n' = |V_M| > n$ and $n'$ is the minimum possible. Let $k = n' - n$. From our choice of $M$, $k \le \frac{n}{2}$.
We construct a multigraph $F$ on $n$ vertices by merging $k$ pairs of vertices of $M$. By \Cref{lem:vertex_merge} this merge does not decrease second eigenvalue. Additionally, after performing all such identifications, the degree of each vertex in $F$ is bounded by $d_F = 2d = 16$.

Now we bound the spectrum of the final graph $G^{(k)}$. By \Cref{lem:patch}:
\begin{equation}
    \lambda_2(G^{(k)}) \geq \frac{1}{2}\lambda_2(F) \geq \frac{d - 5\sqrt{2}}{2} > 0.464.
\end{equation}

For the maximum eigenvalue, using the degree bounds with $d_H = d = 8$ and $d_F = 16$:
\begin{equation}
    \lambda_n(G^{(k)}) \leq 2\left(d_H + \frac{3}{2}d_F + (d_H\tau + 2)\frac{d_F}{\lambda_2(H)} + 1\right) \leq 2\left(8 + 24 + \frac{16(8\tau + 2)}{0.928} + 1\right) < 135 + 276\tau.
\end{equation}

Thus, the condition number of the sequence is bounded by:
\begin{equation}
    \chi_\cG = \frac{\lambda_n(G^{(k)})}{\lambda_2(G^{(k)})} \leq \frac{135 + 276\tau}{0.464} < 291 + 595 \tau.
\end{equation}

Finally, we define $S$ as the union of $s$, its neighbors in $H$, and all stars inserted into all edges incident to $s$ in $H$ by the $\tau$-star extension. Accordingly, $T$ is the union of $t$ and all its neighbors and adjacent stars.
We obtain:
\begin{equation}
  \rho(S, T) \geq \rho(\{s\}, \{t\}) - 2 (\tau + 1) \geq (\tau + 1)(\Delta(H) - 2) > (\tau + 1) (\log m - 2.15).
\end{equation}

\end{proof}

\begin{proof}[Proof of \Cref{thm:main}.]
  The objective is constructed in a standard way by splitting Nesterov's ``worst function'' $f(x)$ zero-chain \cite{carmon2020lower,nesterov2004introduction} into two additive terms \cite{arjevani2015communication}
\begin{align}
  f(x) &= \frac1n\left(\varphi_1(x) + \varphi_2(x)\right),
  \\
  \varphi_1(x) &= \frac{n(L_g - \mu_g)}{8}\cbr{x^\top M_1 x - 2x_1} + \frac{n\mu_g}{4} \sqn{x},
  \\
  \varphi_2(x) &= \frac{n(L_g - \mu_g)}{8}{x^\top M_2 x} + \frac{n\mu_g}{4} \sqn{x},
\end{align}
where
\begin{equation}
  M_1 = \diag{M_0, M_0, \ldots}, \quad M_2 = \diag{1, M_0, M_0, \ldots}, \quad M_0= \begin{pmatrix} 1 & -1 \\ -1 & 1 \end{pmatrix} .
\end{equation}
Both $\varphi_1$ and $\varphi_2$ are $n L_g/ 2$-smooth, $f$ is $\mu_g$-strongly convex, and $L_g$ is the exact constant of Lipschitz smoothness of $f$.

Let $\varphi_1(x)$ and $\varphi_2(x)$ be evenly distributed between two subsets of nodes $S, T \subset V(\cG)$ defined below, i.e.
\begin{equation}
f_i(x) = \begin{cases}
   \frac{1}{|S|}\varphi_1(x), & i \in S,
   \\
   \frac{1}{|T|}\varphi_2(x), & i \in T,
   \\
   0, & \text{otherwise}.
\end{cases}
\end{equation}

Each set $S$ (resp. $T$) consists of $s$ (resp. $t$), its $d$ neighbors in $H$, and the $\tau$ internal vertices of each of the $d$ stars inserted on the incident edges, so that $|S| = |T| = 1 + d + d\tau \geq 2 + d\tau$.
This guarantees that all $f_i$ have Lipschitz constant of gradient being at most
\begin{equation*}
  \frac{nL_g/2}{|S|} \;\leq\; \frac{nL_g/2}{2 + d\tau} \;=\; \frac{m^2(1+4\tau)L_g/2}{2 + 8\tau} \;=\; \frac{m^2 L_g}{4} \;\sref{eq:m}{\leq}\; L_l,
\end{equation*}
where we used $n = m^2(1 + d\tau/2) = m^2(1 + 4\tau)$ with $d = 8$, if we set
  \begin{equation}\label{eq:m}
    m = \floor{\sqrt{\frac{L_l}{L_g}}}.
  \end{equation}
  
By requirement $L_l /L_g \geq 4$ we have $m\geq 2$. 
We take $\tau = \lfloor\frac{\chi - 291}{595}\rfloor$, which is the inverse of the bound $\chi_\cG \le 291 + 595\tau$ of \Cref{lem:distance}: this guarantees $\chi_\cG \leq 291 + 595\lfloor\frac{\chi-291}{595}\rfloor \leq \chi$. Since $\tau + 1 > \frac{\chi - 291}{595}$, \Cref{lem:distance} also gives $\rho(S, T)\geq (\tau+1)(\log m - 2.15) > \frac{\chi - 291}{595}(\log m -2.15)$. The construction requires $\tau \geq 1$, i.e. $\chi_\cG \geq 886$, which holds by the hypothesis of \Cref{thm:main}.

The idea of zero-chain-based lower bounds in convex optimization is to obtain an objective function such that (a) all components of its minimum $x^*$ are nonzero; (b) each oracle call (e.g. gradient computation) can increase the number of nonzero components $q$ at most by $1$.
Since this part of the proof is standard, we simply refer here to the classical result about the properties of Nesterov's zeros chain and its adaptation to distributed algorithms.
By \cite[Theorem~2.1.13]{nesterov2004introduction}
\begin{equation}
  \norm{x^k - x^*} \geq \cbr{\frac{\sqrt{L_g} -\sqrt{\mu_g}}{\sqrt{L_g} + \sqrt{\mu_g}}}^{q} \norm{x^0 - x^*},
\end{equation}
where $q$ is the index of the last nonzero component of $x^k$. 
Assuming w.l.o.g. that $x^0 = 0$, for any first-order distributed algorithm  $q \leq {\floor{\frac{N_\cG}{\rho(S, T)}} + 1}$ \cite[Lemma 1]{scaman2017optimal}.

Thus if $\norm{x^k - x^*}/\norm{x^0 - x^*} \leq \eps$, then
\begin{equation}
  q \geq \log^{-1}\cbr{1 + \frac{2 \sqrt{\mu_g}}{\sqrt{L_g} - \sqrt{\mu_g}}}\log\frac1\eps
  \geq
 \frac{\sqrt{L_g} - \sqrt{\mu_g}}{2 \sqrt{\mu_g}}\log\frac1\eps
  =\frac12 \left(\sqrt{\frac{L_g}{\mu_g}} - 1\right)\log\frac1\eps, \\
\end{equation}
and
\begin{equation}\label{eq:NW}
N_\cG \geq \rho(S,T) \cdot(q - 1) \geq \cbr{\frac{\chi - 291}{595}}\left(\log\cbr{\sqrt{\frac{L_l}{L_g} }- 1 } -2.15\right)\cbr{\frac12 \left(\sqrt{\frac{L_g}{\mu_g}} - 1\right)\log\frac1\eps - 1}.
\end{equation}

It remains to reduce \eqref{eq:NW} to the form stated in \Cref{thm:main}.
Recall that $n = |V_\cG| = m^2(1 + 4\tau)$ is the number of nodes. By the choice $\tau = \lfloor\frac{\chi-291}{595}\rfloor$ we have $595\tau + 291 \leq \chi \leq 595\tau + 886$, so $\chi_\cG = \chi = \Theta(\tau)$ and therefore
\begin{equation}\label{eq:n_over_chi}
  \frac{n}{\chi_\cG} = \frac{m^2(1 + 4\tau)}{\chi} = \Theta(m^2) = \Theta\cbr{\frac{L_l}{L_g}},
\end{equation}
using $m = \floor{\sqrt{L_l/L_g}}$. Consequently $\log\frac{n}{\chi_\cG} = \Theta(\log m) = \Theta\cbr{\log\sqrt{L_l/L_g}}$, which matches the factor $\log(\sqrt{L_l/L_g} - 1) - 2$ in \eqref{eq:NW}. The side condition $n/\chi_\cG \geq 4$ of \Cref{thm:main} is, up to the constants absorbed in \eqref{eq:n_over_chi}, the requirement $m \geq 2$, i.e. $L_l/L_g \geq 4$, which is exactly the regime in which the $\tau$-star extension is nontrivial.
Substituting $\frac{\chi - 291}{595} = \Omega(\chi_\cG)$, $\sqrt{L_g/\mu_g} = \sqrt{\kappa_g}$, and $\log\frac{n}{\chi_\cG} = \Theta\cbr{\log\frac{L_l}{L_g}}$ into \eqref{eq:NW}, we obtain
\begin{equation}
  N_\cG = \Omega\cbr{\chi_\cG \sqrt{\kappa_g}\,\log\frac{n}{\chi_\cG}\,\log\frac1\eps},
\end{equation}
which is the claimed bound.
\end{proof}

\section{Conclusion}
In this paper we showed that, similar to the static setup, simple replacement of $\mu_l$ with $\mu_g$ in the communication complexity of decentralized algorithms cannot be achieved, due to the $\log\frac{L_l}{L_g}$ factor, because the diameter of expander graphs with fixed $\chi_G$ can be as large as $\log(|V_G|)$.
Using the $\tau$-star extension construction and the edge patch procedure we obtained the natural $N_\cG = \Omega(\chi_\cG)$ scaling of communication complexity with the condition number of communication graph.

The main open questions in this research direction are removing the $\log \frac{L_g}{\mu_g}$ term in the best known communication upper bound in static setup (Mudag), and obtaining similar upper bounds in the time-varying setup.

It is also interesting if a similar lower bound to \Cref{thm:main} could be obtained for the slowly time-varying setup, where the number of edges that can be added or removed between communication rounds is limited \cite{metelev2023consensus, metelev2023decentralized}.

\bibliography{references}

@article{hoory2006expander,
  author = {Hoory, S. and Linial, N. and Wigderson, A.},
  title = {Expander Graphs and Their Applications},
  journal = {Bulletin of the American Mathematical Society},
  year = {2006},
  volume = {43},
  number = {4},
  pages = {439--561}
}

@inproceedings{kovalev2021lower,
  author = {Kovalev, D. and Gasanov, E. and Gasnikov, A. and Richtárik, P.},
  title = {Lower Bounds and Optimal Algorithms for Smooth and Strongly Convex Decentralized Optimization Over Time-Varying Networks},
  booktitle = {Advances in Neural Information Processing Systems},
  volume = {34},
  pages = {22325--22335},
  year = {2021}
}

@article{spielman2012spectral,
  title={Spectral graph theory},
  author={Spielman, Daniel},
  journal={Combinatorial scientific computing},
  volume={18},
  number={18},
  year={2012},
  publisher={CRC Press Boca Raton, Florida}
}

@inproceedings{metelev2023consensus,
  title={Is consensus acceleration possible in decentralized optimization over slowly time-varying networks?},
  author={Metelev, Dmitry and Rogozin, Alexander and Kovalev, Dmitry and Gasnikov, Alexander},
  booktitle={International Conference on Machine Learning},
  pages={24532--24554},
  year={2023},
  organization={PMLR}
}

@book{nesterov2004introduction,
	author = {Nesterov, Yurii},
	title = {Introductory Lectures on Convex Optimization: a basic course},
	publisher = {Kluwer Academic Publishers, Massachusetts},
	year = {2004}
}

@inproceedings{scaman2017optimal,
	title={Optimal algorithms for smooth and strongly convex distributed optimization in networks},
	author={Scaman, Kevin and Bach, Francis and Bubeck, S{\'e}bastien and Lee, Yin Tat and Massouli{\'e}, Laurent},
	booktitle={Proceedings of the 34th International Conference on Machine Learning-Volume 70},
	pages={3027--3036},
	year={2017},
	organization={JMLR. org}
}

@article{scaman2019optimal,
  title={Optimal convergence rates for convex distributed optimization in networks},
  author={Scaman, Kevin and Bach, Francis and Bubeck, S{\'e}bastien and Lee, Yin Tat and Massouli{\'e}, Laurent},
  journal={Journal of Machine Learning Research},
  volume={20},
  number={159},
  pages={1--31},
  year={2019}
}

@article{nedic2009distributed,
	title={Distributed subgradient methods for multi-agent optimization},
	author={Nedi{\'c}, Angelia and Ozdaglar, Asuman},
	journal={IEEE Transactions on Automatic Control},
	volume={54},
	number={1},
	pages={48--61},
	year={2009},
	publisher={IEEE}
}

@article{nedic2010constrained,
  title={Constrained consensus and optimization in multi-agent networks},
  author={Nedic, Angelia and Ozdaglar, Asuman and Parrilo, Pablo A},
  journal={IEEE Transactions on Automatic Control},
  volume={55},
  number={4},
  pages={922--938},
  year={2010},
  publisher={IEEE}
}

@incollection{gorbunov2022recent,
  title={Recent theoretical advances in decentralized distributed convex optimization},
  author={Gorbunov, Eduard and Rogozin, Alexander and Beznosikov, Aleksandr and Dvinskikh, Darina and Gasnikov, Alexander},
  booktitle={High-Dimensional Optimization and Probability: With a View Towards Data Science},
  pages={253--325},
  year={2022},
  publisher={Springer}
}

@article{ryabinin2021moshpit,
  title={Moshpit sgd: Communication-efficient decentralized training on heterogeneous unreliable devices},
  author={Ryabinin, Max and Gorbunov, Eduard and Plokhotnyuk, Vsevolod and Pekhimenko, Gennady},
  journal={Advances in Neural Information Processing Systems},
  volume={34},
  pages={18195--18211},
  year={2021}
}

@article{lian2017can,
  title={Can decentralized algorithms outperform centralized algorithms? a case study for decentralized parallel stochastic gradient descent},
  author={Lian, Xiangru and Zhang, Ce and Zhang, Huan and Hsieh, Cho-Jui and Zhang, Wei and Liu, Ji},
  journal={Advances in neural information processing systems},
  volume={30},
  year={2017}
}

@inproceedings{seide20141,
  title={1-bit stochastic gradient descent and its application to data-parallel distributed training of speech DNNs.},
  author={Seide, Frank and Fu, Hao and Droppo, Jasha and Li, Gang and Yu, Dong},
  booktitle={Interspeech},
  volume={2014},
  pages={1058--1062},
  year={2014},
  organization={Singapore}
}

@article{richtarik2021ef21,
  title={EF21: A new, simpler, theoretically better, and practically faster error feedback},
  author={Richt{\'a}rik, Peter and Sokolov, Igor and Fatkhullin, Ilyas},
  journal={Advances in Neural Information Processing Systems},
  volume={34},
  pages={4384--4396},
  year={2021}
}

@article{beznosikov2023biased,
  title={On biased compression for distributed learning},
  author={Beznosikov, Aleksandr and Horv{\'a}th, Samuel and Richt{\'a}rik, Peter and Safaryan, Mher},
  journal={Journal of Machine Learning Research},
  volume={24},
  number={276},
  pages={1--50},
  year={2023}
}

@article{boyd2004fastest,
  title={Fastest mixing Markov chain on a graph},
  author={Boyd, Stephen and Diaconis, Persi and Xiao, Lin},
  journal={SIAM review},
  volume={46},
  number={4},
  pages={667--689},
  year={2004},
  publisher={SIAM}
}

@inproceedings{lakhotia2022polarfly,
  title={Polarfly: A cost-effective and flexible low-diameter topology},
  author={Lakhotia, Kartik and Besta, Maciej and Monroe, Laura and Isham, Kelly and Iff, Patrick and Hoefler, Torsten and Petrini, Fabrizio},
  booktitle={SC22: International Conference for High Performance Computing, Networking, Storage and Analysis},
  pages={1--15},
  year={2022},
  organization={IEEE}
}

@article{liao2025ub,
  title={Ub-mesh: a hierarchically localized nd-fullmesh datacenter network architecture},
  author={Liao, Heng and Liu, Bingyang and Chen, Xianping and Guo, Zhigang and Cheng, Chuanning and Wang, Jianbing and Chen, Xiangyu and Dong, Peng and Meng, Rui and Liu, Wenjie and others},
  journal={IEEE Micro},
  year={2025},
  publisher={IEEE}
}

@article{ying2021exponential,
  title={Exponential graph is provably efficient for decentralized deep training},
  author={Ying, Bicheng and Yuan, Kun and Chen, Yiming and Hu, Hanbin and Pan, Pan and Yin, Wotao},
  journal={Advances in Neural Information Processing Systems},
  volume={34},
  pages={13975--13987},
  year={2021}
}

@article{arjevani2015communication,
  title={Communication complexity of distributed convex learning and optimization},
  author={Arjevani, Yossi and Shamir, Ohad},
  journal={Advances in neural information processing systems},
  volume={28},
  year={2015}
}

@article{ye2023multi,
  title={Multi-consensus decentralized accelerated gradient descent},
  author={Ye, Haishan and Luo, Luo and Zhou, Ziang and Zhang, Tong},
  journal={Journal of machine learning research},
  volume={24},
  number={306},
  pages={1--50},
  year={2023}
}

@article{KovKupRog26,
  title={On the complexity of decentralized optimization via global function parameters},
  author={Kovalev, D. A. and Kupavskii, A. B. and Rogozin, A. V. and Yarmoshik, D. V.},
  journal={Uspekhi Mat. Nauk},
  volume={81},
  number={3(489)},
  pages={163--164},
  year={2026},
  note={\url{http://mi.mathnet.ru/eng/rm10278}}
}

@article{marfoq2020throughput,
  title={Throughput-optimal topology design for cross-silo federated learning},
  author={Marfoq, Othmane and Xu, Chuan and Neglia, Giovanni and Vidal, Richard},
  journal={Advances in Neural Information Processing Systems},
  volume={33},
  pages={19478--19487},
  year={2020}
}

@article{tyurin2026optimality,
  title={Optimality in Decentralized Optimization under Bandwidth Constraints},
  author={Tyurin, Alexander},
  journal={arXiv preprint arXiv:2603.20735},
  year={2026}
}

@article{linial1992locality,
  title={Locality in distributed graph algorithms},
  author={Linial, Nathan},
  journal={SIAM Journal on computing},
  volume={21},
  number={1},
  pages={193--201},
  year={1992},
  publisher={SIAM}
}

@phdthesis{rozhon2024local,
  title={Local Complexity: New Results and Bridges to Other Fields},
  author={Rozhon, Vaclav},
  year={2024},
  school={ETH Zurich}
}

@inproceedings{dubey2020kernel,
  title={Kernel methods for cooperative multi-agent contextual bandits},
  author={Dubey, Abhimanyu and others},
  booktitle={International Conference on Machine Learning},
  pages={2740--2750},
  year={2020},
  organization={PMLR}
}

@article{shihab2026locality,
  title={Locality, Not Spectral Mixing, Governs Direct Propagation in Distributed Offline Dynamic Programming},
  author={Shihab, Ibne Farabi},
  journal={arXiv preprint arXiv:2604.18615},
  year={2026}
}

@article{lazarsfeld2023simple,
  title={Simple Opinion Dynamics for No-Regret Learning},
  author={Lazarsfeld, John and Alistarh, Dan},
  journal={arXiv preprint arXiv:2306.08670},
  year={2023}
}

@article{tyurin2024optimal,
  title={On the optimal time complexities in decentralized stochastic asynchronous optimization},
  author={Tyurin, Alexander and Richt{\'a}rik, Peter},
  journal={Advances in Neural Information Processing Systems},
  volume={37},
  pages={122652--122705},
  year={2024}
}

@phdthesis{koloskova2024optimization,
  title={Optimization algorithms for decentralized, distributed and collaborative machine learning},
  author={Koloskova, Anastasiia},
  year={2024},
  school={EPFL}
}

@inproceedings{yarmoshik2025decentralized,
  title={Decentralized optimization with coupled constraints},
  author={Yarmoshik, Demyan and Rogozin, Alexander and Kiselev, Nikita and Dorin, Daniil and Gasnikov, Alexander and Kovalev, Dmitry},
  booktitle={International Conference on Learning Representations},
  volume={2025},
  pages={63369--63390},
  year={2025}
}

@article{li2024accelerated,
  title={Accelerated gradient tracking over time-varying graphs for decentralized optimization},
  author={Li, Huan and Lin, Zhouchen},
  journal={Journal of Machine Learning Research},
  volume={25},
  number={274},
  pages={1--52},
  year={2024}
}

@article{kovalev2020optimal,
  title={Optimal and practical algorithms for smooth and strongly convex decentralized optimization},
  author={Kovalev, Dmitry and Salim, Adil and Richt{\'a}rik, Peter},
  journal={Advances in Neural Information Processing Systems},
  volume={33},
  pages={18342--18352},
  year={2020}
}

@article{carmon2020lower,
  title={Lower bounds for finding stationary points I},
  author={Carmon, Yair and Duchi, John C and Hinder, Oliver and Sidford, Aaron},
  journal={Mathematical Programming},
  volume={184},
  number={1},
  pages={71--120},
  year={2020},
  publisher={Springer}
}

@article{metelev2023decentralized,
  title={Decentralized optimization over slowly time-varying graphs: Algorithms and lower bounds},
  author={Metelev, Dmitry and Beznosikov, Aleksandr and Rogozin, Alexander and Gasnikov, Alexander and Proskurnikov, Anton},
  journal={arXiv preprint arXiv:2307.12562},
  year={2023}
}

@inproceedings{dwork2006calibrating,
  title={Calibrating noise to sensitivity in private data analysis},
  author={Dwork, Cynthia and McSherry, Frank and Nissim, Kobbi and Smith, Adam},
  booktitle={Theory of cryptography conference},
  pages={265--284},
  year={2006},
  organization={Springer}
}

@article{dwork2014algorithmic,
  title={The algorithmic foundations of differential privacy},
  author={Dwork, Cynthia and Roth, Aaron},
  journal={Foundations and trends{\textregistered} in theoretical computer science},
  volume={9},
  number={3-4},
  pages={211--487},
  year={2014},
  publisher={Emerald Publishing Limited}
}

@inproceedings{allouah2024privacy,
  title={The Privacy Power of Correlated Noise in Decentralized Learning},
  author={Allouah, Youssef and Koloskova, Anastasia and El Firdoussi, Aymane and Jaggi, Martin and Guerraoui, Rachid},
  booktitle={International Conference on Machine Learning},
  pages={1115--1143},
  year={2024},
  organization={PMLR}
}

\end{document}